\documentclass[a4paper,10pt, reqno]{amsart}
\usepackage{a4wide}
\usepackage{amsthm}
\usepackage{amssymb}
 \newtheorem{theorem}{Theorem}[section]
 
 \newtheorem{lemma}[theorem]{Lemma}
 \newtheorem{corollary}[theorem]{Corollary}

 \newtheorem{notations}{Notations}
 \newtheorem{hypothesis}{Hypothesis}

\begin{document}
\title{Moments of the error term in the Sato-Tate law for elliptic curves}
\author{Stephan Baier, Neha Prabhu}
\address{Stephan Baier, Ramakrishna Mission Vivekananda University, Department of Mathematics, PO Belur Math, Dist Howrah 711202,
West Bengal, India}
\email[Stephan Baier]{email$_{-}$baier@yahoo.de}
\address{Neha Prabhu, Queen's University, Kingston, Ontario K7K 3N6}
\email[Neha Prabhu]{neha.prabhu@queensu.ca}
\begin{abstract} We derive new bounds for moments of the error in the Sato-Tate law over families of elliptic curves. Our estimates are
stronger than those obtained in \cite{BZh} and \cite{BSh} for the first and second moment, but this comes at the cost
of larger ranges of averaging. As applications, we
deduce new almost-all results for the said errors and a conditional Central Limit Theorem on the distribution of these errors. 
Our method is different from those used in the above-mentioned papers and builds
on recent work by the second-named author and K. Sinha \cite{PrS} who derived a Central Limit Theorem on the distribution of the errors in the Sato-Tate 
law for families of cusp forms for the full modular group. In addition, identities by Birch and Melzak play a crucial rule in this paper. 
Birch's identities connect moments of coefficients of Hasse-Weil $L$-functions for elliptic curves with the Kronecker class number and further
with traces of Hecke operators. Melzak's identity is combinatorial in nature. 
\end{abstract}

\subjclass[2010]{11G05 (primary); 11G40 (secondary)}
\maketitle

\section{Introduction and statement of results}
For $(a,b)\in \mathbb{Z}\times \mathbb{Z}$ with $\Delta(a,b):=4a^3+27b^2\not=0$ let $E(a,b)$ be the elliptic curve given in Weierstrass form by
$$
y^2=x^3+ax+b.
$$
Consider the Hasse-Weil $L$-function
$$
L(E;s):=\sum\limits_{n=1}^{\infty} a_E(n)n^{-s}=\prod\limits_{p|\mathcal{N}_E} \left(1-a_E(p)p^{-s}\right)^{-1} 
\prod\limits_{p\not\ \ \! \! \! \! |\mathcal{N}_E} \left(1-a_E(p)p^{-s}+p^{1-2s}\right)^{-1} \quad (\Re(s)>1),
$$
where $\mathcal{N}_E$ is the conductor of $E$. By $\tilde{a}_E(n)$ we denote the normalized $n$-th coefficient, given by
$$
\tilde{a}_E(n):=\frac{a_E(n)}{\sqrt{n}}. 
$$ It is due to Hasse that for $p$ prime, $\tilde{a}_E(p) \in [-2,2].$ The distribution of the sequence $\tilde{a}_E(p)$ in the interval $[-2,2]$ has been well investigated in the past few decades. For any interval $I\subseteq [-2,2]$ and elliptic curve $E$, let
$$
N_I(E,x):=\sharp\{x/2< p\le x\ :\ p \mbox{ prime},\  p\nmid\mathcal{N}_E,\ \tilde{a}_E(p)\in I\}, 
$$
where $\mathcal{N}_E$ is the conductor of $E$. Define $\tilde{\pi}(x)$ to be the number of primes between $x/2$ and $x$. 
The Sato-Tate law for elliptic curves, conjectured independently by Sato and Tate around 1960 and recently proved by  L. Clozel, M. Harris, N. Shepherd-Barron and R. Taylor (see \cite{CLT}, \cite{HST} and 
\cite{Tay}), is equivalent to the following 
assertion about the distribution of the $\tilde{a}_E(p)$'s in the interval $[-2,2]$.

\begin{theorem} Let $E$ be an elliptic curve without complex multiplication over $\mathbb{Q}$ and $I$ be a subinterval of 
	$[-2,2]$. Then 
	\begin{equation}\label{Sato-Tate elliptic}
		\lim\limits_{x\rightarrow \infty} \frac{N_I(E,x)}{\tilde{\pi}(x)} = \int\limits_{I} \frac{1}{\pi}\sqrt{1-\frac{t^2}{4}} dt.  
	\end{equation}

\end{theorem}
The Sato-Tate law has since been proved in full generality for Fourier coefficients of modular forms by T. Barnet-Lamb, D. Geraghty, M. Harris and 
R. Taylor (see \cite{BGH}). We denote by $\mu(I)$ the Sato-Tate measure of any subinterval $I\subseteq [-2,2]$ given by the right hand side of \eqref{Sato-Tate elliptic}.

In \cite{BZh}, the first-named author and L. Zhao established results which imply the following bounds for the first and second moments of the error
$N_I(E,x) - \tilde{\pi}(x)\mu(I)$ in the Sato-Tate law over families of elliptic curves. 

\begin{theorem}[Zhao-Baier] \label{ZaBa}  Fix $\varepsilon>0$ and $c>0$. Let $I=[\alpha,\beta]$ be a subinterval of $(0,2]$.
	Suppose that $x^{\varepsilon-5/12}\le (\beta-\alpha)/\beta \le x^{-\varepsilon}$ and $\mu(I)\ge x^{\varepsilon-1/2}$. Then the 
	following hold, where, by convention, the case when $\Delta(a,b)=4a^3+27b^2=0$ is excluded from the summations over $a$ and $b$ below, and the
	$O$-constants depend only on $\varepsilon$ and $c$.\medskip\\
	{\rm (i)} If  
	\begin{equation} \label{thatcanbemadesmaller}
	A,B\ge x^{1/2+\varepsilon} \quad \mbox{and} \quad  AB\ge x^{1+\varepsilon}\mu(I)^{-1},
	\end{equation}
	then
	\begin{equation} \label{Satoave}
	\frac{1}{4AB} \sum\limits_{|a|\le A} \sum\limits_{|b|\le B} \left(N_I(E(a,b),x) - \tilde{\pi}(x)\mu(I)\right)=
	O_{\varepsilon,c}\left(\frac{\tilde{\pi}(x)\mu(I)}{(\log x)^{c}}\right).
	\end{equation}
	{\rm (ii)} If  
	$$
	A,B\ge x^{1+\varepsilon} \quad \mbox{and} \quad x^{2+\varepsilon}\mu(I)^{-2} \le AB\le \exp\left(\exp\left(x^{1-\varepsilon}\right)\right),
	$$
	then
	\begin{equation} \label{Satoavesecond}
	\frac{1}{4AB} \sum\limits_{|a|\le A} \sum\limits_{|b|\le B} \left(N_I(E(a,b),x) - \tilde{\pi}(x)\mu(I)\right)^2=
	O_{\varepsilon,c}\left(\frac{\left(\tilde{\pi}(x)\mu(I)\right)^2}{(\log x)^{c}}\right).
	\end{equation}
\end{theorem}

Throughout the sequel, we want to keep the convention that the case $\Delta(a,b)=0$ is excluded from all summations over $a$ and $b$.  
We note that the summations in \eqref{Satoave} and \eqref{Satoavesecond} include pairs $(a,b)$ such that $E(a,b)$ is a CM-curve, for which the Sato-Tate law 
is known not to hold by a result of M. Deuring \cite{Deu}. However, the set of such pairs $(a,b)$ with $|a|\le A$ and $|b|\le B$
has cardinality $O\left(\min\left\{A^{1/2},B^{1/3}\right\}\right)$ (see \cite{Bai}, for example), and therefore, the contribution of these pairs
is negligible.

The following almost-all result follows immediately from Theorem \ref{ZaBa}(ii) (see also \cite[Corollary 2]{BZh}).

\begin{corollary} \label{almost} Fix $c,d>0$. Then, under the conditions of {\rm Theorem \ref{ZaBa}(ii)}, we have 
	\begin{equation}\label{almostall}
	|N_I(E(a,b),x)-\tilde{\pi}(x)\mu(I)| \ll  \frac{\tilde{\pi}(x)\mu(I)}{(\log x)^{c}}
	\end{equation}
	for all $(a,b)\in \mathbb{Z}^2$ with $|a|\le A$ and $|b|\le B$, except for $O\left(AB(\log x)^{-d}\right)$ pairs $(a,b)$.
\end{corollary}

In \cite{Bai}, it was shown that \eqref{thatcanbemadesmaller} in Theorem \ref{ZaBa}(i) can be replaced by the condition
\begin{equation} \label{abovementioned}
A,B\ge x^{\varepsilon} \quad \mbox{and} \quad  x^{1+\varepsilon}\mu(I)^{-1}\le AB\le x^{F},
\end{equation}
where $F$ is any positive constant. Even stronger bounds were obtained by Banks and Shparlinski \cite{BSh} who 
obtained a power saving of $x^{\delta}$ over the trivial bound  
for the first moment if 
$I$ is fixed and 
\begin{equation} \label{abovementioned1}
A,B\ge x^{\varepsilon} \quad \mbox{and} \quad  AB\ge x^{1+\varepsilon}.
\end{equation}
Here the size of $\delta$ depends on $\varepsilon$ and is smaller than $1/12$ (see equation (20) in \cite{BSh}). 

There are a number of related results in the literature (see, in particular, \cite{DKS} and \cite{ShSh}). 
In this paper, we treat all moments, not only the first and second moments, and obtain new estimates. Our focus lies on {\it strong savings}
over the trivial bounds rather than as small as possible families of curves (as weak as possible conditions on $A$ and $B$), 
which latter was the goal in the papers \cite{BZh} 
and \cite{BSh} as well as subsequent papers on this subject. 
Our savings for the first and second 
moments are indeed stronger than those obtained in \cite{BZh} and \cite{BSh}. In particular, for the first moment, we get, for fixed $I$, 
a saving of $x^{1/4}(\log x)^c$ unconditionally and $x^{1/2-\varepsilon}$ under MRH (a particular case of the Generalized Riemann Hypothesis, 
stated below) as compared to the power of logarithm saving 
in Theorem \ref{ZaBa}(i) and the above-mentioned saving of $x^{\delta}$ with $\delta<1/12$ obtained in 
\cite{BSh}. The price of this improvement will be that 
our families of curves are larger, {\it i.e.}, our conditions on $A$ and $B$ are stronger than those in Theorem \ref{ZaBa}, 
\eqref{abovementioned} and \eqref{abovementioned1}, 
but at a moderate level.  
More generally, we shall obtain power savings over the trivial bound for all moments. 

To describe the results obtained, we require some notation.
By $\sigma_k(T_p)$ we denote the trace and by $\tilde\sigma_k(T_p)$ the normalized trace of the Hecke operator $T_p$, 
acting on the space of cusp forms of weight $k$ for the full modular group, {\it i.e.}, 
$$\tilde{\sigma}_{k}(T_p) = \frac{\sigma_{k}(T_p)}{p^{\frac{k-1}{2}}}.$$
In addition, we state a number of hypotheses below.\\ \\
{\bf Modular Riemann Hypothesis - MRH:} The Riemann Hypothesis holds for all $L$-functions associated to cusp forms $f\in S_k(\Gamma_0(1))$
with $k\in \mathbb{N}$.\\

We note that  
\begin{equation} \label{thumbsup}
	\sum\limits_{x/2<p\le x} \tilde{\sigma}_{k}(T_p) = O\left(kx^{1/2}\log kx\right)
	\end{equation}
using the dimension formula for $S_k(\Gamma_0(1))$ and the generalized prime number theorem (see Lemma \ref{avesig} below). 
To prove {\it asymptotic} estimates rather than just bounds for the moments of the error in the Sato-Tate law and deduce a Central Limit Theorem on the distribution of this error, 
we will need the following plausible hypothesis which is slightly stronger on average over $k$ (namely, by a power of $\log x$). 

\begin{hypothesis} Let $c>0$ and $d_2>d_1>0$ be arbitrary but fixed. Then we have
	\begin{equation} \label{Hypo}
	\sum\limits_{k\le K} \frac{1}{k} \cdot \left| \sum\limits_{x/2<p\le x} \tilde{\sigma}_{k}(T_p) \right| = O_{c,d_1,d_2}\left(Kx^{1/2}(\log x)^{-c}\right) 
	\end{equation}
	as $x\rightarrow\infty$ if $d_1\log x\le \log K \le d_2\log x$.
\end{hypothesis}

We shall also use a second hypothesis which doensn't concern traces of Hecke operators.

\begin{hypothesis} Let $c,d>0$ be arbitrary but fixed and suppose that $m\in \mathbb{N}$. Then we have
	\begin{equation} \label{Hypo2}
	\sum\limits_{y<p\le x} \tilde{a}_{E}(p^m) = O_{c,d}\left(mx(\log x)^{-c}\right) 
	\end{equation}
	as $x\rightarrow\infty$ if $0\le y<x$ and $\log m \le d\log (\mathcal{N}_Ex)$.
\end{hypothesis}

The above Hypothesis 2 is true under Langland's conjectures 
(see \cite{Lan}), which themselves imply the Sato-Tate law. To see this, one applies \cite[Theorem 5.15]{IwK}, the generalized prime
number theorem, to the symmetric power $L$-functions associated to $E$, which are automorphic and hence entire under the said conjectures, 
and uses the multiplicative properties of the coefficients
$\tilde{a}_E(n)$ (see \cite{Murty}, for example). 

Now we are ready to state our new moment bounds.

\begin{theorem} \label{momenttheorem}
	Fix $c,\varepsilon>0$ and $t\in \mathbb{N}$. Set
	\begin{equation*} 
	\eta(t):=\max\{t,2(t-1)\}
	\end{equation*}
	and 
	\begin{equation} \label{deltadef}
	\delta(t):=\begin{cases} 1 & \mbox{ if } t \mbox{ is even}\\ 0 & \mbox{ if } t \mbox{ is odd.} \end{cases}
	\end{equation}
	Suppose that $A,B\ge 1$ such that $AB\le \exp\left(x^{1/2-\varepsilon}\right)$.
	Then we have 
	\begin{equation}\label{momentbounds}
	\begin{split}
	& \frac{1}{4AB} \sum\limits_{|a|\le A} \sum\limits_{|b|\le B} \left(N_I(E(a,b),x)-\tilde{\pi}(x)\mu(I)\right)^t \\ 
	= & \delta(t)\cdot \frac{t!}{2^{t/2}(t/2)!}\cdot \left(\mu(I)-\mu(I)^2\right)^{t/2} \cdot \tilde{\pi}(x)^{t/2} +\\ 
	& \begin{cases} O_{t,c,\varepsilon}\left(x^{3t/4}(\log x)^{-c}\right) & \mbox{unconditionally if } 
	A,B\ge x^{\eta(t)+\varepsilon},\\
	O_{t,\varepsilon}\left(x^{t/2}(\log x)^{t/2}\right) & \mbox{under {\rm MRH} if } A,B\ge x^{3\eta(t)/2+\varepsilon},\\
	O_{t,c,\varepsilon}\left(x^{t/2}(\log x)^{-c}\right) & \mbox{under {\rm Hypotheses 1,2} if } A,B\ge x^{3\eta(t)/2+\varepsilon},\\
	O_{t,c,\varepsilon}\left(x^{t/2}(\log x)^{-c}\right) & \mbox{under {\rm Hypothesis 1} if } A,B\ge x^{2\eta(t)+\varepsilon}.
	\end{cases}
	\end{split}
	\end{equation}
	The fourth bound above under {\rm Hypothesis 1} holds without assuming $AB\le \exp\left(x^{1/2-\varepsilon}\right)$.
\end{theorem}

We point out that the main term on the right-hand side of \eqref{momentbounds} is dominated by the $O$-terms in the first two estimates, the unconditional
one and the one under MRH, but not by the $O$-term in the third and fourth estimates under Hypotheses 1,2 if $I$ is not too short.

To achieve these results, we use a method which is different from those in \cite{BZh} and \cite{BSh}, where the key point was the use 
of multiplicative characters
to detect isomorphism classes of curves modulo primes. Our approach builds instead on the 
work \cite{PrS} by the second-named author and K. Sinha about the distribution of the error in the Sato-Tate law for modular forms. 
Here the starting point is to detect the condition that $\tilde{a}_E(p)\in I$ by employing Theorem \ref{errorapprox} below, which
was established in \cite{PrS} in the context of Fourier coefficients of cusp forms 
using Beurling-Selberg polynomials and the multiplicative properties of the coefficients in question. Then we use identities by 
Birch which connect moments of the coefficients $a_{E(a,b)}(p)$ with the Kronecker class
number and further with traces of Hecke operators (see sections 6 and 7). This is followed by an application of 
Melzak's identity which is combinatorial in nature (see section 8). 
In this way, we connect two different kinds of families - families
of elliptic curves and families of Hecke eigenforms for the full modular group. 

Multiplicative characters are also applied in a similar fashion 
as in \cite{BZh} (see section 5). This, however, is not essential for obtaining the savings in our estimates but only for 
lowering the sizes of our families of elliptic curves (see the remarks at the beginning of section 6). There may be some hope that these
sizes can be reduced further by employing some ideas from \cite{Bai} or \cite{BSh}. 

As applications, we deduce new almost-all results which give support to a conditional estimate by K. Murty \cite{KMu} 
and a conjecture by S. Akiyama and Y. Tanigawa \cite{AkT} for individual curves. Moreover, we derive a 
Central Limit Theorem on the distribution of the error in the Sato-Tate law, conditional under Hypothesis 1.
The said almost-all result is as follows and can be immediately deduced from the estimates for the second moment (case $k=2$) in Theorem \ref{momenttheorem}.

\begin{corollary} \label{almostallnew} Fix $c,\varepsilon>0$. Suppose that $A,B\ge 1$ such that $AB\le \exp\left(x^{1/2-\varepsilon}\right)$ and $y>1$.
	Then for all pairs $(a,b)\in \mathbb{Z}^2$ with $|a|\le A$ and $|b|\le B$ with the 
	exception of $O_{c,\varepsilon}\left(ABy^{-2}\right)$ pairs, we have 
	\begin{equation} \label{errorbound}
	\begin{split}
	& |N_I(E(a,b),x)-\tilde{\pi}(x)\mu(I)|\\ 
	\le & \begin{cases} yx^{3/4}(\log x)^{-c} & \mbox{ unconditionally if } A,B\ge x^{2+\varepsilon} \\
	yx^{1/2}(\log x)^{1/2} & \mbox{ under {\rm MRH} if } A,B\ge x^{3+\varepsilon}\\
	y\left(\left(\mu(I)-\mu(I)^2\right)^{1/2}\tilde{\pi}(x)^{1/2}+x^{1/2}(\log x)^{-c}\right)
	& \mbox{ under {\rm Hypotheses 1,2} if } A,B\ge x^{3+\varepsilon}\\
	y\left(\left(\mu(I)-\mu(I)^2\right)^{1/2}\tilde{\pi}(x)^{1/2}+x^{1/2}(\log x)^{-c}\right)
	& \mbox{ under {\rm Hypothesis 1} if } A,B\ge x^{4+\varepsilon}.
	\end{cases}
	\end{split}
	\end{equation}
	The fourth bound above under {\rm Hypothesis 1} holds without assuming $AB\le \exp\left(x^{1/2-\varepsilon}\right)$.
\end{corollary}

K. Murty \cite{KMu} proved that 
$$
N_I(E,x)-\tilde{\pi}(x)\mu(I) \ll x^{3/4}(\log \mathcal{N}_Ex)^{1/2}
$$
for every non-CM curve $E$, where $\mathcal{N}_E$ is the conductor of $E$, if all symmetric power $L$-functions associated to $E$ are automorphic and
satisfy the Riemann Hypothesis. The first, unconditional, estimate in \eqref{errorbound}
gives support towards this conditional bound. It even shows that we
have a slightly stronger bound for, in a sense, almost all curves $E$ (take, for example, $y:=(\log x)^{c/2}$).

A conjecture by S. Akiyama and Y. Tanigawa \cite{AkT} 
(see also the survey paper \cite{Maz}) suggests that the bound
\begin{equation} \label{AkTa}
|N_I(E,x)-\tilde{\pi}(x)\mu(I)| \ll_E x^{1/2+\varepsilon}
\end{equation}
should hold for all non-CM curves $E$,
and there is numerical
evidence in favor of it. The conditional estimates in \eqref{errorbound}  give support towards this conjecture 
(take, for example, $y:=x^{\varepsilon/2}$). 
Moreover, the observation that the term $\left(\mu(I)-\mu(I)^2\right)^{1/2}\tilde{\pi}(x)^{1/2}$ in the third and fourth estimates cannot be removed 
gives rise to the conjecture that the exponent $1/2+\varepsilon$ in \eqref{AkTa} 
is essentially optimal, which is supported by numerical data as well (see \cite{Maz}). 

As a second application of Theorem \ref{momenttheorem} the following Central Limit Theorem can be deduced from the last estimate 
for the moments in \eqref{momentbounds} under Hypothesis 1 by adapting the method of moments used in \cite{PrS}.  

\begin{theorem} \label{main} Suppose that $A=A(x)\ge 1$ and $B=B(x)\ge 1$ satisfy 
	$\frac{\log A}{\log x},\frac{\log B}{\log x}\rightarrow \infty \mbox{ as } x \rightarrow \infty$. 
	Assume that {\rm Hypothesis 1} holds.   
	Then for any bounded continuous real function $h$ on $\mathbb{R}$, we have 
	\begin{equation*} 
	\lim\limits_{x\rightarrow \infty} \frac{1}{4AB} \sum\limits_{|a|\le A}\sum\limits_{|b|\le B} h
	\left(\frac{N_I(E(a,b),x)-\tilde{\pi}(x) \mu(I)}{\sqrt{\tilde{\pi}(x)\left(\mu(I)-\mu(I)^2\right)}}\right) =
	\frac{1}{\sqrt{2\pi}} \int\limits_{-\infty}^{\infty} h(t)e^{-t^2/2}\ dt.
	\end{equation*}
\end{theorem}

This corresponds to the following unconditional Central Limit Theorem for the error in the Sato-Tate law 
for families of modular forms, established in \cite{PrS}. To understand the result, we first set up some notations.
For any $N,k\in \mathbb{N}$ let $\mathcal{F}_{N,k}$ be an orthonormal basis of the subspace of all newforms in the space $S_k(\Gamma_0(N))$ of cusp forms
of weight $k$ with respect to $\Gamma_0(N)$. \medskip\\
{\rm (i)} For any $f\in \mathcal{F}_{N,k}$, let
$$
f(z):=\sum\limits_{n=1}^{\infty} a_f(n)q^n \quad \mbox{with } q=e^{2\pi i z}
$$
be its Fourier expansion. By $\tilde{a}_f(n)$ we denote the normalized $n$-th coefficient, given by
$$
\tilde{a}_f(n):=\frac{a_f(n)}{n^{(k-1)/2}}. 
$$
{\rm (ii)} For any interval $I\subseteq [-2,2]$ and $f\in \mathcal{F}_{N,k}$, let
$$
N_I(f,x):=\sharp\{p\le x\ :\ p \mbox{ prime},\  p\nmid N,\ \tilde{a}_f(p)\in I\}.
$$ The said Central Limit Theorem established in \cite{PrS} is as follows.
\begin{theorem}[Prabhu-Sinha] \label{mainmodular} Suppose that $k=k(x)$ satisfies $\frac{\log k}{\sqrt{x}\log x} \rightarrow \infty$ 
	as $x\rightarrow \infty$. Then for any bounded continuous real function $h$ on $\mathbb{R}$, we have 
	\begin{equation*} 
	\lim\limits_{x\rightarrow \infty} \frac{1}{\sharp \mathcal{F}_{1,k}} \sum\limits_{f\in \mathcal{F}_{1,k}} h
	\left(\frac{N_I(f,x)-\pi(x) \mu(I)}{\sqrt{\pi(x)\left(\mu(I)-\mu(I)^2\right)}}\right) =
	\frac{1}{\sqrt{2\pi}} \int\limits_{-\infty}^{\infty} h(t)e^{-t^2/2}\ dt.
	\end{equation*}
\end{theorem}

We now turn to the tools used to prove our main result, Theorem \ref{momenttheorem}.
A key result in \cite[Section 4]{PrS} that discusses the case of the first moment is the following approximation of the error in the Sato-Tate law for the case of Fourier coefficients of cusp forms. It will be essential in this work as well.
\begin{theorem} \label{errorapprox} Let $M\in \mathbb{N}$ and $I=[2\cos \beta,2\cos \alpha]\subseteq [-2,2]$, where $0\le \alpha<\beta\le \pi$.
	Then there exist real numbers $U_I^-(1),...,U_I^-(M),U_I^+(1),...,U_I^+(M)$ such that the following hold, where all $O$-constants below are
	absolute. \medskip\\
	{\rm (i)} We have 
	\begin{equation}\label{Udef}
	U_I^{\pm}(m)= \begin{cases}
	S_I^{\pm}(m)-S_I^{\pm}(m+2)  & \mbox{if } m\le M-2\\
	S_I^{\pm}(m)  & \mbox{if } m\in \{M-1,M\},
	\end{cases}
	\end{equation}
	where 
	$$
	S_I^{\pm}(m):=\frac{\sin(2\pi m\beta)-\sin(2\pi m\alpha)}{m\pi}+O\left(\frac{1}{M}\right).
	$$
	{\rm (ii)} Set 
	$$
	P_I^{\pm}(E,x):= \sum\limits_{1\le m\le M}
	U_I^{\pm}(m)\sum\limits_{\substack{x/2<p\le x\\ p\nmid \mathcal{N}_E}} \tilde{a}_E(p^m).
	$$
	Then
	\begin{equation} \label{ap}
	P_I^{-}(E,x)+O\left(\frac{\tilde{\pi}(x)}{M}\right) \le 
	N_I(E,x) - \tilde{\pi}(x)\mu(I) \le P_I^+(E,x)+ O\left(\frac{\tilde{\pi}(x)}{M}\right).
	\end{equation} 
\end{theorem}

An amazing and very useful fact, worked out in \cite{PrThesis}, is that the sum of the squares of
coefficients $U_I^{\pm}(m)$ above can
be approximated using an expression depending on the Sato-Tate measure. This is the content of the following theorem.

\begin{theorem} \label{amazing} Let $M\ge 1$ and $U_I^{\pm}(m)$ be defined as in Theorem {\rm \ref{errorapprox}} above. Then
	\begin{equation}
	\sum\limits_{1\le m\le M} U_I^{\pm}(m)^2=\mu(I)-\mu(I)^2 + O\left(\dfrac{\log (2M)}{M}\right)
	\end{equation}
\end{theorem}  

Now to prove Theorem \ref{momenttheorem}, it shall suffice to establish the following two moment bounds. 

\begin{theorem} \label{T2} 
Fix $t\in \mathbb{N}$ and define $\delta(t)$ as in \eqref{deltadef}.
Assume that $U(m)_{m\in \mathbb{N}}$ is a sequence of complex numbers such that
$$
U(m)\ll \frac{1}{m} \quad \mbox{for all } m\in \mathbb{N}.
$$
Let $M\ge 1$ and set 
\begin{equation} \label{Zdef}
Z:=\sum\limits_{1\le m\le M} U(m)^2.
\end{equation}
Fix $F,c,\varepsilon>0$. Suppose that $A,B\ge 1$ satisfy $AB\le \exp\left(x^{1/2-\varepsilon}\right)$. Then we have 
\begin{equation} \label{das}
\begin{split}
& \frac{1}{4AB}  \sum\limits_{|a|\le A} \sum\limits_{|b|\le B} \Bigg(\sum\limits_{1\le m\le M} U(m) \sum\limits_{\substack{x/2<p\le x\\
		p\nmid ab\Delta(a,b)}} 
\tilde{a}_{E(a,b)}(p^m) \Bigg)^t\\ = & \delta(t)\cdot \frac{t!}{2^{t/2}(t/2)!}\cdot Z^{t/2} \left(\tilde{\pi}(x)^{t/2}+
O\left(\tilde{\pi}(x)^{t/2-1}\right)\right) + \\ &  
\begin{cases}
O_{t,F,c,\varepsilon}\left(M^tx^{t/2}(\log x)^{-c}\right) & \mbox{unconditionally if } 
x^{\varepsilon} \le M\le x^F \mbox{ and } A,B\ge x^{t+\varepsilon}\\
O_{t,F,\varepsilon}\left(M^t(\log x)^t\right) & \mbox{under {\rm MRH} if } \tilde{\pi}(x)^{1/2}\le M\le x^F \mbox{ and } A,B\ge 
x^{3t/2+\varepsilon}\\
O_{t,F,c,\varepsilon}\left(M^t(\log x)^{-c}\right) & \mbox{under {\rm Hyp.1,2}  if }
\tilde{\pi}(x)^{1/2}\le M\le x^F \mbox{ and } A,B\ge x^{3t/2+\varepsilon}.
\end{cases}
\end{split}
\end{equation}
\end{theorem}

\begin{theorem} \label{T3} 
Under the conditions of {\rm Theorem \ref{T2}} with the condition $AB\le \exp\left(x^{1/2-\varepsilon}\right)$ omitted, we have 
\begin{equation} \label{dasi}
\begin{split}
& \frac{1}{4AB}  \sum\limits_{|a|\le A} \sum\limits_{|b|\le B} \Bigg(\sum\limits_{1\le m\le M} U(m) \sum\limits_{\substack{x/2<p\le x\\
		p\nmid \Delta(a,b)}} 
\tilde{a}_{E(a,b)}(p^m) \Bigg)^t\\ = & \delta(t)\cdot \frac{t!}{2^{t/2}(t/2)!}\cdot Z^{t/2}\left(\tilde{\pi}(x)^{t/2} + 
O\left(\tilde{\pi}(x)^{t/2-1}\right)\right)+
O_{t,F,c,\varepsilon}\left(M^t(\log x)^{-c}\right)
\end{split}
\end{equation}
if $\tilde{\pi}(x)^{1/2}\le M\le x^F$ and $A,B\ge x^{2t+\varepsilon}$, provided that {\rm Hypothesis 1} holds.  
\end{theorem}

We note that on the left-hand side of \eqref{dasi}, the summation condition $p\nmid ab$, which is present in \eqref{das}, is omitted. 
Avoiding this summation condition 
comes at the cost of a stronger condition on  $A$ and $B$ in Theorem \ref{T3}, as compared to Theorem \ref{T2}, but on the other hand, we don't need to assume
the truth of Hypothesis 2, and the condition $AB\le \exp\left(x^{1/2-\varepsilon}\right)$ is not needed either. 

Again, we point out that the main term on the right-hand side of \eqref{das} is dominated by the $O$-terms in the first two estimates, 
the unconditional one and the one under MRH, but not necessarily by the $O$-terms in the third estimate in \eqref{das} and in \eqref{dasi}.

Our strategy of proof of Theorems \ref{T2} and \ref{T3} will be roughly as follows. First,
we open up the $t$-th power. Then we reduce the products $\tilde{a}_{E(a,b)}(p_1^{m_1})\cdots \tilde{a}_{E(a,b)}(p_t^{m_t})$ arising in 
this way to linear combinations 
of terms of the form $\tilde{a}_{E(a,b)}(n)$, where the prime divisors of $n$ belong to the set $\{p_1,...,p_t\}$. 
Now we pull in the sums over $a$ and $b$ and evaluate the averages of $\tilde{a}_{E(a,b)}(n)$
over $a$ and $b$. It turns out that they can be approximated using a multiplicative function $S(n)$ if $A$ and $B$ are large enough. This function
$S(n)$ will be investigated further. Since it is multiplicative, it suffices to compute it at prime powers. To this end, we use identities by
Birch and Melzak. Finally, we exploit the averaging over the primes $p_1$,...,$p_t$ and the natural numbers $m_1,...,m_t$. 
The main term will come from the contribution of $S(1)$. 

In the following sections \ref{prelim} to \ref{aver}, we will provide the results that we need for the final proofs of Theorems \ref{T2} and
\ref{T3}, which will
be carried out in sections \ref{finalproof} and \ref{finalproof1}, respectively.\\ 

{\bf Acknowledgements:} We would like to thank Brundaban Sahu (NISER Bhubaneshwar) for having made SAGE computations which suggested the correctness of Lemma \ref{Alk},
providing a proof of Lemma \ref{Alk} for the case $l=0$ and making us aware of the reference \cite{AMM}, which we used to prove the said lemma
in full generality. The first-named author would like to thank the School of Physical Sciences at JNU Delhi for a great time.  The second named author would like to thank IISER Pune for its resources and the National Board for Higher Mathematics for the PhD scholarship.

\section{Preliminaries} \label{prelim}
\subsection{Notations and basic facts}
The following notations will be used throughout this paper. 
\begin{notations}
{\rm (i)} We reserve the symbol $p$ for primes greater or equal 5 and the symbol $E$ for elliptic curves over $\mathbb{Q}$, and we denote by $\mathcal{N}_E$
the conductor of $E$.\medskip\\
{\rm (ii)} Throughout this paper, we assume that $x\ge 10$ and write
\begin{equation} \label{tildepidef}
\tilde{\pi}(x):=\pi(x)-\pi\left(\frac{x}{2}\right)=\sharp\{p \ \mbox{prime}: x/2<p\le x\}.
\end{equation}
{\rm (iii)} Throughout this paper, we denote by $I$ an arbitrary but fixed subinterval of $[-2,2]$ and 
\begin{equation} \label{muIdef}
\mu(I):= \int\limits_{I} \frac{1}{\pi}\sqrt{1-\frac{t^2}{4}}\ dt
\end{equation}
\end{notations}

We recall the following well-known facts on coeffients of Hasse-Weil 
$L$-functions associated to elliptic curves $E$ over $\mathbb{Q}$, which will be of key importance for our work (see \cite{Sil}, for example). 

\begin{theorem} \label{facts}
{\rm (i)} We have
$$
a_E(p)=\begin{cases} p+1-\sharp E_p \mbox{ if } E \mbox{ has good reduction at } p\\
       \in \{-1,0,1\} \mbox{ otherwise,}
       \end{cases}
$$
where $E_p$ is the curve over $\mathbb{F}_p$ obtained by reducing $E$ modulo $p$. \medskip\\
{\rm (ii)} For every elliptic curve $E$ and every $n\in \mathbb{N}$, we have
$$
|\tilde{a}_E(n)|\le d(n),
$$
where $d(n)$ is the number of divisors of $n$. In particular, if $p$ is a prime, then
$$
\tilde{a}_E(p)\in [-2,2].
$$
{\rm (iii)} The arithmetic functions $a_E : \mathbb{N} \rightarrow \mathbb{R}$ and $\tilde{a}_E : \mathbb{N} \rightarrow \mathbb{R}$ are 
multiplicative.\medskip\\
{\rm (iv)} For any prime $p$ at which $E$ has good reduction and any non-negative integers $i$ and $j$, we have 
$$
\tilde{a}_E(p^i)\tilde{a}_E(p^j)=\sum\limits_{l=0}^{\min(i,j)} \tilde{a}_E(p^{i+j-2l}).
$$
{\rm (v)} For any prime $p$ at which $E$ has bad reduction and any non-negative integer $i$, we have 
$$
\tilde{a}_E(p^i)= \tilde{a}_E(p)^i.
$$
\end{theorem}

We shall also use the following well-known dimension formula for the space $S_k(\Gamma_0(1))$ of cusp forms for the full modular group
in the course of this paper (see \cite{Lang}, for example). 

\begin{theorem} \label{dim} Let $k\in \mathbb{N}$. Then 
$$
\mbox{dim } S_k(\Gamma_0(1)) = \begin{cases} 0 & \mbox{if } k \mbox{ is odd} \\ 
                                             0 & \mbox{if } k=2 \\
                                             \left\lfloor\frac{k}{12}\right\rfloor & \mbox{if } k \mbox{ is even and } k\not\equiv 2 \bmod{12}\\
                                             \left\lfloor\frac{k}{12}\right\rfloor -1 & \mbox{ if } k>2 \mbox{ and } k\equiv 2 \bmod{12},
                               \end{cases} 
$$
where for $z\in \mathbb{R}$, $\lfloor z \rfloor$ denotes the largest integer not exceeding $z$. 
\end{theorem}

\subsection{Averages of traces of Hecke operators}
In our paper, we shall establish a connection between families of elliptic curves and traces of Hecke operators. Recall 
\begin{equation} \label{trdef}
\sigma_{k}(T_p)=\sum\limits_{f\in \mathcal{F}_{1,k}} a_f(p) \quad \mbox{and} \quad 
\tilde{\sigma}_{k}(T_p)=\sum\limits_{f\in \mathcal{F}_{1,k}} \tilde{a}_f(p).
\end{equation}
We begin with collecting estimates for averages of these traces. Unconditionally,
we have the following.

\begin{lemma} \label{avesiguncond} Suppose that $k\in \mathbb{N}$. Then  
\begin{equation} \label{unconditionally} 
\sum\limits_{x/2<p\le x} \tilde{\sigma}_{k}(T_p) = O\left(kx(\log kx)^4\exp\left(-C\sqrt{\log x}\right)\right), 
\end{equation}
where $C>0$ is a suitable constant.
\end{lemma}

\begin{proof} Let $\mathcal{F}_{1,k}$ be the orthonormal basis of Hecke eigenforms of weight $k$ for the
full modular group. Then, by \cite[Theorem 5.13]{IwK} (generalized prime number theorem), we have
$$
\sum\limits_{x/2<p\le x} \tilde{a}_f(p) = O\left(x(\log kx)^4\exp\left(-C\sqrt{\log x}\right)\right) \quad \mbox{ if } f\in 
\mathcal{F}_{1,k}
$$
for some constant $C>0$, where the $O$-constant is absolute. This implies
$$
\sum\limits_{x/2< p\le x} \tilde{\sigma}_{k}(T_p) = \sum\limits_{f\in \mathcal{F}_{1,k}} \sum\limits_{x/2<p\le x} 
\tilde{a}_f(p) \ll 
kx(\log kx)^4\exp\left(-C\sqrt{\log x}\right)
$$
using Theorem \ref{dim}.
\end{proof}

We shall also need the following bound with the same kind of saving by a factor of $\exp\left(-C\sqrt{\log x}\right)$ for the 
average of the product of two traces of Hecke operators. 

\begin{lemma} \label{avetwosiguncond} Suppose that $k,l\in \mathbb{N}$. Then  
\begin{equation*} 
\sum\limits_{x/2<p\le x} \tilde{\sigma}_{k}(T_p)\tilde{\sigma}_{l}(T_p) = O\left(klx(\log klx)^4\exp\left(-C\sqrt{\log x}\right)\right), 
\end{equation*}
where $C>0$ is a suitable constant.
\end{lemma}

\begin{proof}
Let $\mathcal{F}_{1,k}$ and $\mathcal{F}_{1,l}$ be the orthonormal bases of Hecke eigenforms of weight $k$ and 
$l$ for the
full modular group, respectively. Applying \cite[Theorem 5.13]{IwK} to
the $L$-function associated to the Rankin-Selberg convolution $f\otimes g$ of 
$f\in \mathcal{F}_{1,k}$ and $g\in \mathcal{F}_{1,l}$, we have
$$
\sum\limits_{x/2<p\le x} \tilde{a}_f(p)\tilde{a}_g(p) = O\left(x(\log klx)^4\exp\left(-C\sqrt{\log x}\right)\right)
$$
for some constant $C>0$, where the $O$-constant is absolute. This implies
$$
\sum\limits_{x/2<p\le x} \tilde{\sigma}_{k}(T_p)\tilde{\sigma}_{l}(T_p) = \sum\limits_{f\in \mathcal{F}_{1,k}} 
\sum\limits_{g\in \mathcal{F}_{1,l}}\sum\limits_{x/2<p\le x} 
\tilde{a}_f(p)\tilde{a}_g(p) \ll klx(\log klx)^4\exp\left(-C\sqrt{\log x}\right)
$$
using Theorem \ref{dim}.
\end{proof}

Under MRH, the following bound holds.

\begin{lemma} \label{avesig} Suppose that $k\in \mathbb{N}$. Then, under {\rm MRH}, we have  
	\begin{equation} \label{underRH}
	\sum\limits_{x/2<p\le x} \tilde{\sigma}_{k}(T_p) = O\left(kx^{1/2}\log kx\right), 
	\end{equation}
	where the $O$-constant is absolute.
\end{lemma}

\begin{proof} The proof follows the same lines as the proof of Lemma \ref{avesiguncond} above, but here we use \cite[Theorem 5.15]{IwK} 
	instead of \cite[Theorem 5.13]{IwK}. 
\end{proof}
 
\section{Identities involving prime powers} \label{ident}
Theorem \ref{facts}(iv) contains an identity which allows to write products of the form $\tilde{a}_E\left(p^i\right)\tilde{a}_E\left(p^j\right)$
as sums of terms of the form $\tilde{a}_E(p^m)$. The following Lemma provides a general result of this kind for products of the form
$\tilde{a}_E\left(p^{m_1}\right)\cdots \tilde{a}_E\left(p^{m_r}\right)$ which was established in \cite{PrS}. It will be used in the beginning
of the proof of Theorem \ref{T2}.

\begin{lemma} \label{prodpolys} Assume that $m_1,...,m_r\in \mathbb{N}$ and $E$ has good reduction at $p$. 
Let $s=m_1+...+m_r$. Then
\begin{equation} \label{products}
\prod\limits_{i=1}^r \tilde{a}_E\left(p^{m_i}\right) = \sum\limits_{m=0}^{\infty} D(m_1,...,m_r;m)\tilde{a}_E(p^m),
\end{equation}
where $D(m_1,...,m_r;m)$ are nonnegative integers satisfying
\begin{equation} \label{Dbounds}
\begin{split}
D(m_1,...,m_r,m)= & 0 \quad \mbox{ if } m>s, \\
D(m_1,...,m_r;m)= & O\left(s^{r-2}\right) \quad \mbox{ if } r\ge 2 \mbox{ and } 1\le m\le s,\\
D(m_1,...,m_r;0)= & O\left(s^{r-3}\right) \quad \mbox{ if } r\ge 3,\\
D(m_1,m_2;0) = & \begin{cases} 1 & \mbox{ if } m_1=m_2\\ 0 & \mbox{ if } m_1\not=m_2, \end{cases}\\
D(m_1;m)= & \begin{cases} 1 & \mbox{ if } m_1=m\\ 0 & \mbox{ if } m_1\not=m.\end{cases} 
\end{split}
\end{equation}
\end{lemma}

Further, it will be useful to express $\tilde{a}_E(p^m)$ as a polynomial in $\tilde{a}_E(p)$.

\begin{lemma} \label{poly} Assume that $m\in \mathbb{N}\cup \{0\}$ and $E$ has good reduction at $p$. Define
\begin{equation} \label{fdef}
f_m(x):=\sum\limits_{j=0}^{\lfloor m/2\rfloor} (-1)^j \binom{m-j}{j} x^{m-2j},
\end{equation}
where we set
$$
\binom{0}{0}=1.
$$
Then
\begin{equation} \label{iter}
\tilde{a}_E(p^m)=f_m\left(\tilde{a}_E(p)\right).
\end{equation}
\end{lemma}

\begin{proof} We prove this lemma by induction over $m$. For $m=0,1$, \eqref{iter} holds trivially. Assume that \eqref{iter} holds for $m=k$.
We show that \eqref{iter} then holds for $m=k+2$. 

By Theorem \ref{facts}(iv), we have 
$$
\tilde{a}_E(p^k)\tilde{a}_E(p^2)=\tilde{a}_E(p^{k+2})+\tilde{a}_E(p^k)+\tilde{a}_E(p^{k-2}).
$$
Hence,
\begin{equation*}
\begin{split}
\tilde{a}_E(p^{k+2})=  \tilde{a}_E(p^k)\tilde{a}_E(p^2)-\tilde{a}_E(p^k)-\tilde{a}_E(p^{k-2})
=  \tilde{a}_E(p^k)\left(\tilde{a}_E(p^2)-1\right)-\tilde{a}_E(p^{k-2}).
\end{split}
\end{equation*}
Further, 
$$
\tilde{a}_E(p)^2=\tilde{a}_E(p^2)+1.
$$
It follows that
$$
\tilde{a}_E(p^{k+2})= \tilde{a}_E(p^k)\left(\tilde{a}_E(p)^2-2\right)-\tilde{a}_E(p^{k-2}).
$$
By induction hypothesis, this implies that
\begin{equation*}
\begin{split}
\tilde{a}_E(p^{k+2})= & \left(\tilde{a}_E(p)^2-2\right)\cdot \sum\limits_{j=0}^{\lfloor k/2 \rfloor} (-1)^j \binom{k-j}{j} \tilde{a}_E(p)^{k-2j}
-\sum\limits_{j=0}^{\lfloor k/2-1\rfloor} (-1)^j \binom{k-2-j}{j} \tilde{a}_E(p)^{k-2-2j}\\
= & \sum\limits_{j=0}^{\lfloor k/2+1\rfloor} (-1)^j\left(\binom{k-j}{j}+2\binom{k-(j-1)}{j-1}-\binom{k-j}{j-2}\right) 
\tilde{a}_E(p)^{k+2-2j}\\
= & \sum\limits_{j=0}^{\lfloor k/2+1\rfloor} (-1)^j\binom{k+2-j}{j}\tilde{a}_E(p)^{k+2-2j},
\end{split}
\end{equation*}
which completes the proof.
\end{proof}

We have the following bound which is consistent with Theorem \ref{facts}(ii). 

\begin{lemma} \label{polybound} Let $m\in \mathbb{N}\cup \{0\}$ and $-2\le x\le 2$. Then 
\begin{equation} \label{mbound}
\left|f_m(x)\right|\le m+1.
\end{equation}
\end{lemma}

\begin{proof} The proof can be done directly, but an indirect argument based on the results we already stated seems the shortest. 
Fix an elliptic curve $E$. The coefficients $\tilde{a}_E(p)$ are known to satisfy the Sato-Tate law as $p$ varies over the
primes of good reduction. In particular, the set 
$$
\left\{\tilde{a}_E(p)\ :\ p \mbox{ prime, } p\nmid \mathcal{N}_E\right\}
$$
is dense in $[-2,2]$. The claim now follows from Theorem \ref{facts}(ii), \eqref{iter} and the continuity of $f_m$. 
\end{proof}

\section{Multiplicative structure of averages} 
In this section, we exhibit that averages of $\tilde{a}_{E(a,b)}(n)$ over pairs $(a,b)$ in a box can be approximated using
a multiplicative function in $n$. Our first result is the following approximation, which will later be refined. 

\begin{lemma} \label{av} For all $A,B\ge 1$ and $n\in \mathbb{N}$,
\begin{equation} \label{weakav} 
\mathop{\sum\limits_{|a|\le A} \sum\limits_{|b|\le B}}_{(ab\Delta(a,b),n)=1} \tilde{a}_{E(a,b)}(n) = 4AB S(n) + 
O\left(d(n)s(n)^2\right)+O\left(d(n)s(n)(A+B)\right),  
\end{equation}
where $s(n)$ is the largest squarefree number dividing $n$, and 
\begin{equation} \label{Sn}
S(n):=\frac{1}{s(n)^2} \mathop{\sum\limits_{a=1}^{s(n)} \sum\limits_{b=1}^{s(n)}}_{(ab\Delta(a,b),n)=1} \tilde{a}_{E(a,b)}(n).
\end{equation}
\end{lemma}

\begin{proof} First, we recall the inequality 
$$
|\tilde{a}_{E(a,b)}(n)|\le d(n)
$$
from Theorem \ref{facts}(ii). It follows that 
\begin{equation} \label{aswell}
|S(n)|\le d(n)
\end{equation}
as well. 

We observe that $\tilde{a}_{E(a,b)}(n)$ 
is doubly periodic in $a$ and $b$ with period $s(n)$ as is seen as follows.   
Since $\tilde{a}_{E(a,b)}(p^m)$ equals a polynomial in $\tilde{a}_{E(a,b)}(p)$ by Lemma \ref{poly}, 
and $\tilde{a}_{E(a,b)}(p)$ is periodic in $a$ and $b$ with
period $p$, respectively, it follows that $\tilde{a}_{E(a,b)}(p^m)$ is also periodic in $a$ and $b$ with period $p$, respectively. Since
$\tilde{a}_{E(a,b)}(n)$ is multiplicative in $n$, we deduce that $\tilde{a}_{E(a,b)}(n)$ is periodic in $a$ and $b$ with period $s(n)$, 
respectively. 

It follows that 
\begin{equation}
\begin{split}
& \mathop{\sum\limits_{|a|\le A} \sum\limits_{|b|\le B}}_{(ab\Delta(a,b),n)=1} \tilde{a}_{E(a,b)}(n)\\ = & \mathop{\sum\limits_{-s(n) 
\left\lfloor \frac{A}{s(n)}\right\rfloor <a\le s(n)\left\lfloor\frac{A}{s(n)}\right\rfloor} \
\sum\limits_{-s(n)
\left\lfloor\frac{B}{s(n)}\right\rfloor <b\le s(n)\left\lfloor\frac{B}{s(n)}\right\rfloor}}_{(ab\Delta(a,b),n)=1} \tilde{a}_{E(a,b)}(n) + 
O\left(d(n)s(n)(s(n)+A+B)\right)
\\
= & 4 \left\lfloor\frac{A}{s(n)}\right\rfloor \left\lfloor\frac{B}{s(n)}\right\rfloor s(n)^2 S(n)  + 
O\left(d(n)s(n)(s(n)+A+B)\right)\\
= & 4ABS(n) +O\left(d(n)s(n)^2\right)+O\left(d(n)s(n)(A+B)\right),
\end{split}
\end{equation}
which completes the proof.
\end{proof}

Moreover, we prove the following.

\begin{lemma} \label{mult} The function $S(n)$ defined in \eqref{Sn} is multiplicative.
\end{lemma}

\begin{proof}
Let $n_1,n_2\in \mathbb{N}$ such that $(n_1,n_2)=1$. Then, writing 
$$
a=a_1s(n_2)+a_2s(n_1) \quad \mbox{and} \quad b=b_1s(n_2)+b_2s(n_1),
$$
we have
\begin{equation*}
\begin{split}
S(n_1n_2)= & \frac{1}{s(n_1n_2)^2} \mathop{\sum\limits_{a=1}^{s(n_1n_2)} \sum\limits_{b=1}^{s(n_1n_2)}}_{(ab\Delta(a,b),n_1n_2)=1} 
\tilde{a}_{E(a,b)}(n_1n_2)\\
= & \frac{1}{s(n_1)^2s(n_2)^2} \mathop{\sum\limits_{a_1=1}^{s(n_1)}\sum\limits_{a_2=1}^{s(n_2)} 
\sum\limits_{b_1=1}^{s(n_1)} \sum\limits_{b_2=1}^{s(n_2)}}_{(ab\Delta(a,b),n_1n_2)=1}  
\tilde{a}_{E(a,b)}(n_1)\cdot  
\tilde{a}_{E(a,b)}(n_2)\\
= & \frac{1}{s(n_1)^2s(n_2)^2} \mathop{\sum\limits_{a_1=1}^{s(n_1)}\sum\limits_{a_2=1}^{s(n_2)} 
\sum\limits_{b_1=1}^{s(n_1)} \sum\limits_{b_2=1}^{s(n_2)}}_{(a_1b_1\Delta(a_1,b_1),n_1)=1=(a_2b_2\Delta(a_2,b_2),n_2)} \tilde{a}_{E(a_1s(n_2),b_1s(n_2))}(n_1)
\cdot \tilde{a}_{E(a_2s(n_1),b_2s(n_1))}(n_2)\\
= & \left(\frac{1}{s(n_1)^2}  \mathop{\sum\limits_{a_1=1}^{s(n_1)} 
\sum\limits_{b_1=1}^{s(n_1)}}_{(a_1b_1\Delta(a_1,b_1),n_1)=1} \tilde{a}_{E(a_1s(n_2),b_1s(n_2))}(n_1)\right) \times\\ & \left(\frac{1}{s(n_2)^2}
\mathop{\sum\limits_{a_2=1}^{s(n_2)} \sum\limits_{b_2=1}^{s(n_2)}}_{(a_2b_2\Delta(a_2,b_2),n_2)=1} 
\tilde{a}_{E(a_2s(n_1),b_2s(n_1))}(n_2)\right)\\
= & S(n_1)S(n_2),
\end{split}
\end{equation*}
which proves the claim.
\end{proof}

\section{A refined average estimate}
Now we improve on the error term in Lemma \ref{av}, getting rid of the first $O$-term in \eqref{weakav}
and saving a factor of $s(n)^{1/2-\varepsilon}$ in the second one. 
We mention that Lemma \ref{av} would be already sufficient to prove a version of Theorem \ref{momenttheorem}, but with stronger 
conditions on $A$ and $B$. In particular, using Lemma \ref{av} in our method, we can establish the first two estimates in \eqref{momentbounds}
with the stronger conditions 
$$
\begin{cases} A,B\ge x^{3\eta(t)/2+\varepsilon} & \mbox{unconditionally}\\ 
A,B\ge x^{2\eta(t)}(\log x)^{-t/2} & \mbox{under MRH}
\end{cases}
$$
on $A$ and $B$. However, we are not content with these results and go for as weak as possible 
conditions on $A$ and $B$. 
To this end, we adapt the methods in \cite{Ba0} and \cite{BZh}, where refined asymptotic 
estimates for expressions of the form
$$
\frac{1}{4AB} \sum\limits_{r\in \mathcal{M}} \mathop{\sum\limits_{|a|\le A} \sum\limits_{|b|\le B}}_{a_{E(a,b)}(p)=r} 1  
$$
were established for $p$ prime and $\mathcal{M}$ a suitable set. Here we extend these considerations to arbitrary
integers $n$ in place of primes $p$. We establish the following theorem. 
Since our proof follows closely the methods used in \cite{Ba0} and \cite{BZh}, we will cut some details.

\begin{theorem} \label{improvedav}
Let $\varepsilon>0$ be arbitrary but fixed. Then for all $A,B\ge 1$ and odd $n\in \mathbb{N}$, we have
\begin{equation} \label{thebound} 
\mathop{\sum\limits_{|a|\le A} \sum\limits_{|b|\le B}}_{(ab\Delta(a,b),n)=1} \tilde{a}_{E(a,b)}(n) = 4AB S(n) + 
O_{\varepsilon,t}\left(d(n)s(n)^{1/2+\varepsilon}(A+B)\right),  
\end{equation}
where $s(n)$ and $S(n)$ are defined as in Lemma {\rm \ref{av}}.
\end{theorem}

\begin{proof}
We first observe that the assertion is trivial if $n=1$. Therefore, we assume that $n>1$ throughout this proof and write the prime factorization
of $n$ as $n=p_1^{m_1}\cdots p_t^{m_t}$.

\subsection{Rewriting in terms of character sums}
Using parts (ii) and (iii) of Theorem \ref{facts} and Lemma \ref{poly} into account, we write
\begin{equation} \label{rewrite}
\begin{split}
& \mathop{\sum\limits_{|a|\le A} \sum\limits_{|b|\le B}}_{(ab\Delta(a,b),n)=1} \tilde{a}_{E(a,b)}(n)\\
 = & \sum\limits_{-2\sqrt{p_1}\le r_1\le 2\sqrt{p_1}}\cdots  \sum\limits_{-2\sqrt{p_t}\le r_t\le 2\sqrt{p_t}}
 f_{m_1}\left(\frac{r_1}{\sqrt{p_1}}\right)\cdots f_{m_t}\left(\frac{r_t}{\sqrt{p_t}}\right) 
 \mathop{\sum\limits_{|a|\le A} \sum\limits_{|b|\le B}}_{\substack{(ab\Delta(a,b),n)=1\\ a_{E(a,b)}(p_j)=r_j {\scriptsize \rm \
 for\ }  j\in \{1,...,t\}}} 1.
 \end{split}
\end{equation}
For $p\nmid \Delta(a,b)$, let $E_p(a,b)$ be the elliptic curve over $\mathbb{F}_p$ obtained by reducing $E(a,b)$ modulo $p$.  
Similarly as in \cite{Ba0}, we divide the inner-most double sum on the right-hand side of \eqref{rewrite} into sums over 
isomorphism classes by writing
\begin{equation} \label{furtherwrite}
\begin{split}
\mathop{\sum\limits_{|a|\le A} \sum\limits_{|b|\le B}}_{\substack{(ab\Delta(a,b),n)=1\\ a_{E(a,b)}(p_j)=r_j {\scriptsize \rm \
 for\ }  j\in \{1,...,t\}}} 1 = \sum\limits_{i_1=1}^{\mathcal{I}(p_1;r_1)} \cdots \sum\limits_{i_t=1}^{\mathcal{I}(p_t;r_t)} 
 \mathop{\sum\limits_{|a|\le A} \sum\limits_{|b|\le B}}_{\substack{E_{p_1}(a,b)\cong E_{p_1}\left(u(p_1;r_1,i_1),v(p_1;r_1,i_1)\right)\\ \cdots\\
 E_{p_t}(a,b)\cong E_{p_t}\left(u(p_t;r_t,i_t),v(p_t;r_t,i_t)\right)}} 1, 
 \end{split}
\end{equation}
where $\mathcal{I}(p_j;r_j)$ are positive integers satisfying
\begin{equation} \label{Kr}
\mathcal{I}(p_j;r_j)\le H(r_j^2-4p_j) \ll p_j^{1/2+\varepsilon},
\end{equation}
$H(r_j^2-4p_j)$ being the Kronecker class number,   
$(u(p_j;r_j,i_j),v(p_j;r_j,i_j))$ are suitable pairs of integers representing isomorphism classes and coprime to $p_j$, and $\cong$ indicates isomorphy of curves over $\mathbb{F}_p$. 
In \cite{Ba0}, we used the fact that isomorphy of curves $E_p(a,b)$ and $E_p(u,v)$ over $\mathbb{F}_p$ can be described using congruence 
relations modulo $p$ involving the parameters $a,b,u,v$ (see Lemma 4 in \cite{Ba0}) and detected these relations using Dirichlet characters 
modulo $p$ (see equation (4.2) and the following equation in \cite{Ba0}). Applying this
treatment to \eqref{furtherwrite}, we get
\begin{equation} \label{aftertreat}
\begin{split}
& \mathop{\sum\limits_{|a|\le A} \sum\limits_{|b|\le B}}_{\substack{(ab\Delta(a,b),n)=1\\ a_{E(a,b)}(p_j)=r_j {\scriptsize \rm \
 for\ }  j\in \{1,...,t\}}} 1\\ = & \sum\limits_{i_1=1}^{\mathcal{I}(p_1;r_1)} \cdots \sum\limits_{i_t=1}^{\mathcal{I}(p_t;r_t)} 
 \sum\limits_{|a|\le A} \sum\limits_{|b|\le B} \prod\limits_{j=1}^t F_{p_j}\left(a\cdot\overline{u(p_j;r_j,i_j)},b\cdot\overline{v(p_j;r_j,i_j)}\right),
 \end{split}
\end{equation}
where $\overline{z}$ is a multiplicative inverse of $z$ modulo $p_j$, i.e., $z\overline{z}\equiv 1 \bmod p_j$ if $p_j\nmid z$, and
$$
F_p(c,d):= \begin{cases} \frac{1}{4\varphi(p)} \sum\limits_{k=1}^4 \left(\frac{c}{p}\right)_4^k  \sum\limits_{\chi \bmod{p}}
           \chi\left(c^3\overline{d}^2\right) & \mbox{if } p\equiv 1 \bmod{4}\\
           \frac{1}{4\varphi(p)} \left(1+\left(\frac{c}{p}\right)_2\right) \left(1+\left(\frac{d}{p}\right)_2\right) \sum\limits_{\chi \bmod{p}}
           \chi\left(c^3\overline{d}^2\right) & \mbox{if } p\equiv 3 \bmod{4},
           \end{cases}
$$
$(\cdot/p)_w$ being the $w$-th power residue symbol.  

\subsection{Division into main and error terms}
At this point, we follow the method in \cite{Ba0}, where we treated only the case $t=1$ of one prime. We will therefore be brief at some places. Similarly as in
\cite{Ba0} and \cite{BZh}, we only deal with the case when $p_1,...,p_t$ are all congruent 1 modulo 4. The general case can be handled similarly,
but we need to divide into more character sums. If we are in the said case $p_1,...,p_t\equiv 1 \bmod{4}$, then using the Chinese Remainder
Theorem, we can simplify \eqref{aftertreat} into 
\begin{equation} \label{keepitsimple}
\begin{split}
\mathop{\sum\limits_{|a|\le A} \sum\limits_{|b|\le B}}_{\substack{(ab\Delta(a,b),n)=1\\ a_{E(a,b)}(p_j)=r_j {\scriptsize \rm \
 for\ }  j\in \{1,...,t\}}} 1 = & \frac{1}{4^{\omega(n)}\varphi(s(n))} \sum\limits_{i=1}^{\mathcal{I}(n;r_1,...,r_t)} 
 \sum\limits_{\substack{\tilde{\chi}\bmod{s(n)}\\ \mbox{\scriptsize \rm ord}(\tilde{\chi})|4}} \sum\limits_{\chi \bmod{s(n)}} 
 \sum\limits_{|a|\le A} \sum\limits_{|b|\le B} \\ & \tilde{\chi}\left(a\overline{u_i(n;r_1,...,r_t)}\right)
 \chi\left(a^3\overline{b}^2\overline{u_i(n;r_1,...,r_t)}^3v_i(n;r_1,...,r_t)^2\right),
 \end{split}
\end{equation}
where 
\begin{equation} \label{w}
\mathcal{I}(n;r_1,...,r_t)=\prod\limits_{j=1}^t \mathcal{I}(p_j;r_j)
\end{equation}
and $(u_i(n;r_1,...,r_t),v_i(n;r_1,...,r_t))$ are suitable pairs of integers. Similarly as in \cite{Ba0} and \cite{BZh}, we divide the right-hand side of \eqref{keepitsimple}
into a main and error term, where the main term is the contribution of characters $\tilde{\chi}$ and $\chi$ such that 
$\tilde{\chi}\chi^3=\chi_0=\chi^2$, $\chi_0$ being the principal character modulo $p$, and the error term is the remaining contribution. 
Using \eqref{rewrite} and \eqref{keepitsimple}, it follows that
\begin{equation} \label{divide}
\mathop{\sum\limits_{|a|\le A} \sum\limits_{|b|\le B}}_{(ab\Delta(a,b),n)=1} \tilde{a}_{E(a,b)}(n)\\
 = M(n;A,B)+E(n;A,B),
\end{equation}
where
\begin{equation} \label{mainterm}
\begin{split}
& M(n;A,B):= \frac{1}{4^{\omega(n)}\varphi(s(n))} \Bigg(\mathop{\sum\limits_{\tilde{\chi}\bmod{s(n)}} \sum\limits_{\chi \bmod{s(n)}}}_
 {\substack{\mbox{\scriptsize \rm ord}(\tilde{\chi})|4\\ \tilde{\chi}\chi^3=\chi_0=\chi^2}} 1\Bigg) \cdot 
 \Bigg(\mathop{\sum\limits_{|a|\le A} \sum\limits_{|b|\le B}}_{(ab,n)=1} 1\Bigg)\times\\ & 
 \left(\sum\limits_{-2\sqrt{p_1}\le r_1\le 2\sqrt{p_1}}\cdots  
 \sum\limits_{-2\sqrt{p_t}\le r_t\le 2\sqrt{p_t}}
 f_{m_1}\left(\frac{r_1}{\sqrt{p_1}}\right)\cdots f_{m_t}\left(\frac{r_t}{\sqrt{p_t}}\right) \sharp\mathcal{I}(n;r_1,...,r_t) \right)
\end{split}
\end{equation}
and
\begin{equation} \label{errorterm}
\begin{split}
& E(n;A,B):= \frac{1}{4^{\omega(n)}\varphi(s(n))} \mathop{\sum\limits_{\tilde{\chi}\bmod{s(n)}} \sum\limits_{\chi \bmod{s(n)}}}_
 {\substack{\\ \mbox{\scriptsize \rm ord}(\tilde{\chi})|4\\ \tilde{\chi}\chi^3\not=\chi_0 \ \mbox{\scriptsize \rm or } \chi^2\not=\chi_0}} \\ & 
 \left(\sum\limits_{-2\sqrt{p_1}\le r_1\le 2\sqrt{p_1}}\cdots  
 \sum\limits_{-2\sqrt{p_t}\le r_t\le 2\sqrt{p_t}}
 f_{m_1}\left(\frac{r_1}{\sqrt{p_1}}\right)\cdots f_{m_t}\left(\frac{r_t}{\sqrt{p_t}}\right)\right.\times\\ & \left.
 \sum\limits_{i=1}^{\mathcal{I}(n;r_1,...,r_t)} \overline{\tilde{\chi}}\overline{\chi}^3\left(u_i(n;r_1,...,r_t)\right) 
 \chi^2\left(v_i(n;r_1,...,r_t)\right)\right)\times \\ & 
 \left(\sum\limits_{|a|\le A} \tilde{\chi}\chi^3(a)\right)\left( \sum\limits_{|b|\le B} \overline{\chi}^2(b)\right).
\end{split}
\end{equation}

\subsection{Treatment of the main term}
We first relate $M(n;A,B)$ to the main term on the right-hand side \eqref{weakav}. This is based on two observations. Firstly,
\begin{equation} \label{obs1}
\begin{split}
M(n;A,B)= & \Bigg(\mathop{\sum\limits_{|a|\le A} \sum\limits_{|b|\le B}}_{(ab,n)=1} 1\Bigg)
\Bigg(\mathop{\sum\limits_{|a|\le s(n)} \sum\limits_{|b|\le s(n)}}_{(ab,n)=1} 1\Bigg)^{-1} M(n;s(n),s(n))\\ = & \frac{1}{4\varphi(s(n))^2}
\Bigg(\mathop{\sum\limits_{|a|\le A} \sum\limits_{|b|\le B}}_{(ab,n)=1} 1\Bigg)
M(n;s(n),s(n)),
\end{split}
\end{equation}
and secondly,
\begin{equation} \label{obs2}
E(n;s(n),s(n))=0
\end{equation}
by the orthogonality relations for Dirichlet characters. From the definition of $S(n)$ in \eqref{Sn}
and the equations \eqref{divide} and \eqref{obs2} above, it follows that
\begin{equation} \label{ha}
4s(n)^2S(n)=\mathop{\sum\limits_{|a|\le s(n)} \sum\limits_{|b|\le s(n)}}_{(ab\Delta(a,b),n)=1} \tilde{a}_{E(a,b)}(n) = M(n;s(n),s(n)).
\end{equation}
From \eqref{obs1} and  \eqref{ha}, we conclude that
\begin{equation} \label{hahaha}
M(n;A,B)=\left(\frac{s(n)}{\varphi(s(n))}\right)^2
\Bigg(\mathop{\sum\limits_{|a|\le A} \sum\limits_{|b|\le B}}_{(ab,n)=1} 1\Bigg)S(n)=
\left(\frac{n}{\varphi(n)}\right)^2
\Bigg(\mathop{\sum\limits_{|a|\le A} \sum\limits_{|b|\le B}}_{(ab,n)=1} 1\Bigg)S(n).
\end{equation}
Detecting the summation condition $(ab,n)=1$ using the M\"obius function, it is easily seen that
\begin{equation} \label{Mobi}
\mathop{\sum\limits_{|a|\le A} \sum\limits_{|b|\le B}}_{(ab,n)=1} 1 = 4AB\left(\frac{\varphi(n)}{n}\right)^2 +
O\left(2^t(A+B)+4^t\right).
\end{equation}
Putting \eqref{hahaha} and \eqref{Mobi} together, and using \eqref{aswell}, we obtain
\begin{equation} \label{maintermasymp}
M(n;A,B)=4ABS(n)+O_t\left(d(n)\left(\frac{n}{\varphi(n)}\right)^2\left(A+B\right)\right).
\end{equation}

\subsection{Treatment of the error term}
Set 
\begin{equation} \label{Indef}
\mathcal{I}_n:= \sum\limits_{-2\sqrt{p_1}\le r_1\le 2\sqrt{p_1}}\cdots  
 \sum\limits_{-2\sqrt{p_t}\le r_t\le 2\sqrt{p_t}} \mathcal{I}(n;r_1,...,r_t).
\end{equation}
For the error term $E(n;A,B)$, defined in \eqref{errorterm}, we employ the same method as in \cite{BZh} based on the 
Polya-Vinogradov inequality and bounds for the second and fourth moments of character sums, getting 
\begin{equation} \label{Eb}
\begin{split}
E(n;A,B) = O\left(s(n)^{\varepsilon}d(n)\left(\mathcal{I}_n(A+B)s(n)^{-1/2}+(\mathcal{I}_nAB)^{1/2}\right)\right),
\end{split}
\end{equation}
where we also use the bound 
$$
\left|f_{m_1}\left(\frac{r_1}{\sqrt{p_1}}\right)\cdots f_{m_t}\left(\frac{r_t}{\sqrt{p_t}}\right)\right|\le (m_1+1)\cdots (m_t+1)=d(n)
$$
following from Lemma \ref{polybound}. Combining \eqref{Kr}, \eqref{w} and \eqref{Indef}, we have
\begin{equation*}
\mathcal{I}_n=O_{\varepsilon}\left(s(n)^{1+\varepsilon}\right),  
\end{equation*}
and hence, taking into account that $(AB)^{1/2}\ll A+B$, we deduce from \eqref{Eb} that
\begin{equation} \label{errortermest}
\begin{split}
E(n;A,B) = O_{\varepsilon,t}\left(d(n)s(n)^{1/2+\varepsilon}(A+B)\right).
\end{split}
\end{equation}
Now the claimed asymptotic estimate \eqref{thebound} follows from \eqref{divide}, \eqref{maintermasymp}, \eqref{errortermest}
and 
$$
\frac{n}{\varphi(n)}\ll_{\varepsilon,t} s(n)^{\varepsilon}.
$$
This completes the proof.
\end{proof}

\section{Relation to the Kronecker class number} 
Next, we evaluate the multiplicative function $S(n)$, exhibited in the last two sections, at prime powers. Recalling its definition in 
\eqref{Sn} and our general condition $p\ge 5$,  we write
\begin{equation} \label{Sdiv}
S\left(p^m\right)=S_0\left(p^m\right)-S_1\left(p^m\right)-S_2\left(p^m\right), 
\end{equation}
where
\begin{equation} \label{S0def}
S_0\left(p^m\right):=\frac{1}{p^2} \mathop{\sum\limits_{a=1}^{p} \sum\limits_{b=1}^{p}}_{p\nmid \Delta(a,b)} \tilde{a}_{E(a,b)}(p^m),
\end{equation}
\begin{equation} \label{S12def}
S_1\left(p^m\right):=\frac{1}{p^2} \sum\limits_{a=1}^{p-1} \tilde{a}_{E(a,0)}(p^m) \quad \mbox{and} \quad
S_2\left(p^m\right):=\frac{1}{p^2} \sum\limits_{b=1}^{p-1} \tilde{a}_{E(0,b)}(p^m).
\end{equation}
We first deal with $S_0\left(p^m\right)$ and handle $S_{1,2}\left(p^m\right)$ later.

Using Lemma \ref{poly}, we obtain
\begin{equation} \label{Spp}
\begin{split}
S_0\left(p^m\right)= & \frac{1}{p^2} \sum\limits_{j=0}^{[m/2]} (-1)^j \binom{m-j}{j} 
\mathop{\sum\limits_{a=1}^{p} \sum\limits_{b=1}^{p}}_{p\not \ \! |\Delta(a,b)} \tilde{a}_{E(a,b)}(p)^{m-2j}\\
= & \frac{1}{p^2} \sum\limits_{j=0}^{[m/2]} (-1)^j \binom{m-j}{j} \cdot
\frac{1}{p^{m/2-j}} \mathop{\sum\limits_{a=1}^{p} \sum\limits_{b=1}^{p}}_{p\not \ \! |\Delta(a,b)} a_{E(a,b)}(p)^{m-2j}.
\end{split}
\end{equation}
Considering the arguments in \cite{Bir}, the following expression for the double sum over $a$ and $b$ in the last line holds. 

\begin{lemma} \label{Dapalem} For any prime $p\ge 5$ and positive integer $g$,  
\begin{equation} \label{DaPa}
\mathop{\sum\limits_{a=1}^{p} \sum\limits_{b=1}^{p}}_{p\not \ \! |\Delta(a,b)} a_{E(a,b)}(p)^{g}= \frac{p-1}{2} \sum\limits_{|r|\le 2\sqrt{p}} r^{g} H(r^2-4p), 
\end{equation}
where 
$H(r^2-4p)$ is the Kronecker class number, and the $O$-constant is absolute.
\end{lemma} 

If $g$ is odd, this gives
\begin{equation*}
\mathop{\sum\limits_{a=1}^{p} \sum\limits_{b=1}^{p}}_{p\not \ \! |\Delta(a,b)} a_{E(a,b)}(p)^{g} = 0,
\end{equation*}
which implies 
\begin{equation} \label{oddcase}
S_0\left(p^m\right)=0 \quad \mbox{if } m \mbox{ is odd.}
\end{equation}
If $m=2k$ is even, then using \eqref{Spp} and \eqref{DaPa}, we obtain 
\begin{equation} \label{Spkave}
\begin{split}
S_0\left(p^{2k}\right):= & 
\frac{1-1/p}{p^{k+1}} \sum\limits_{j=0}^{k} (-1)^j \binom{2k-j}{j} p^{j}\cdot \frac{1}{2} \sum\limits_{|r|\le 2\sqrt{p}} r^{2(k-j)} H(r^2-4p)\\
= & \frac{1-1/p}{p^{k+1}} \sum\limits_{j=0}^{k} (-1)^{k-j} \binom{k+j}{k-j} p^{k-j}\cdot \frac{1}{2} \sum\limits_{|r|\le 2\sqrt{p}} 
r^{2j} H(r^2-4p). 
\end{split}
\end{equation}

\section{An identity by Birch} 
Now we use the following identity due to Birch \cite{Bir}.

\begin{lemma} \label{birch} For every prime $p\ge 5$ and positive integer $j$, 
\begin{equation} \label{bi} 
\begin{split}
\frac{1}{2} \sum\limits_{|r|\le 2\sqrt{p}} r^{2j} H(r^2-4p)= &
\frac{(2j)!}{j!(j+1)!}\cdot p^{j+1}\\ & -\sum\limits_{l=1}^j (2l+1)\cdot \frac{(2j)!}{(j-l)!(j+l+1)!}
\cdot p^{j-l}\left(\sigma_{2(l+1)}(T_p)+1\right),
\end{split}
\end{equation}
where $\sigma_{2(l+1)}(T_p)$ is the trace of the Hecke operator $T_p$ acting on the space of cusp forms of weight $2(l+1)$ for the full
modular group. 
\end{lemma}

Plugging \eqref{bi} into the last line of \eqref{Spkave} and re-arranging summations gives
\begin{equation} \label{plugging}
\begin{split}
S_0\left(p^{2k}\right)= &
\frac{1-1/p}{p^{k+1}} \sum\limits_{j=0}^{k} (-1)^{k-j}\binom{k+j}{k-j} p^{k-j}\times\\
& \left(\frac{(2j)!}{j!(j+1)!}\cdot p^{j+1}-\sum\limits_{l=1}^j (2l+1)\cdot \frac{(2j)!}{(j-l)!(j+l+1)!}
\cdot p^{j-l}\left(\sigma_{2(l+1)}(T_p)+1\right)\right)\\
= & \left(1-\frac{1}{p}\right)\left(A_{0,k} - \sum\limits_{l=1}^k A_{l,k}p^{-(l+1)}\left(\sigma_{2(l+1)}(T_p)+1\right)\right),
\end{split}
\end{equation}
where 
\begin{equation}
A_{l,k}:= (2l+1) \sum\limits_{j=l}^k (-1)^{k-j}\binom{k+j}{k-j}\cdot \frac{(2j)!}{(j-l)!(j+l+1)!}. 
\end{equation}

\section{An identity by Melzak}
We claim the following.

\begin{lemma} \label{Alk} For any nonnegative integers $k,l$ with $0\le l\le k$, we have 
$$
A_{l,k}=\begin{cases} 0 & \mbox{ if } l<k \\ 1 & \mbox{ if } l=k. \end{cases} 
$$
\end{lemma}

\begin{proof} It is clear that this holds if $l=k$. For the case when $l<k$, we use Melzak's identity (see \cite{AMM}) which states that
\begin{equation} \label{melzak}
f(x+y)=x\binom{x+n}{n} \sum\limits_{a=0}^n (-1)^a\binom{n}{a}\cdot \frac{f(y-a)}{x+a}
\end{equation}
for all polynomials of degree up to $n$ and $x\not\in \{0,-1,...,-n\}$. Rearranging factors, and making a change of
variables $n=k+l$ and $a=j+l$, it is easily seen that
\begin{equation*}
\begin{split}
A_{l,k}= & (2l+1) \sum\limits_{j=l}^k (-1)^{k-j} \binom{k+j}{k+l} \binom{k+l}{j+l} \cdot \frac{1}{j+l+1}\\
= & (2l+1)\sum\limits_{a=2l}^n (-1)^{n-a} \binom{n+a-2l}{n} \binom{n}{a} \cdot \frac{1}{a+1}.
\end{split}
\end{equation*}
Now we set $x=1$, $y=0$ and 
$$
f(z):=\frac{(n-z-2l)(n-1-z-2l)\cdots (1-z-2l)}{n!}.
$$
Then it follows that
$$
A_{l,k}= (2l+1)(-1)^n \sum\limits_{a=0}^n (-1)^a\binom{n}{a}\cdot \frac{f(y-a)}{x+a},
$$
and \eqref{melzak} therefore gives
\begin{equation}
\begin{split}
A_{l,k}= & (2l+1)(-1)^n \cdot \frac{f(x+y)}{x\binom{x+n}{n}}=(2l+1)(-1)^n \cdot \frac{f(1)}{n+1}\\ = & 
(2l+1)(-1)^n \cdot \frac{(n-1-2l)(n-2-2l)\cdots (-2l)}{(n+1)!}=0
\end{split}
\end{equation}
since $n-1-2l=k-l-1\ge 0$. This completes the proof.
\end{proof}

Now, combining \eqref{oddcase}, \eqref{plugging} and Lemma \ref{Alk}, we obtain the following. 

\begin{lemma} \label{putting} For any prime $p\ge 5$ and $m\in \mathbb{N}$, we have 
$$
S_0(p^m) = \left(1-p^{-1}\right)p^{-(m/2+1)}\sigma_{m+2}(T_p),
$$
where $\sigma_{m+2}(T_p)=0$ if $m$ is odd.
\end{lemma}

\section{Averages over prime powers} \label{aver}
Now we want to bound averages of $S\left(p^m\right)$ and, more generally, products of the form 
$S\left(p^{m_1}\right)\cdots S\left(p^{m_r}\right)$ over primes. To this end, we first handle the functions $S_{1,2}\left(p^m\right)$,
defined in \eqref{S12def}, which is easy. 

\begin{lemma} \label{S12lemma} Let $c,d>0$ be arbitrary but fixed and $m\in \mathbb{N}$. Then the following hold.\medskip\\
{\rm (i)} We have 
$$
S_{1,2}\left(p^m\right)=O\left( \frac{m}{p}\right).
$$
{\rm (ii)} Under {\rm Hypothesis 2}, we have
$$
\sum\limits_{x/2<p\le x} S_{1,2}\left(p^m\right)=O_{c,d}\left( \frac{m}{(\log x)^c}\right) 
$$
if $\log m\le d\log x$.
\end{lemma}

\begin{proof}
Part (i) is a direct consequence of Theorem \ref{facts}(ii), and part (ii) follows from Hypothesis 2 and partial summation 
after re-arranging summations in the form
$$
\sum\limits_{x/2<p\le x} S_{1}\left(p^m\right) = \sum\limits_{a=1}^{x-1} \sum\limits_{\max\{x/2,a\}<p\le x} \frac{\tilde{a}_{E(a,0)}(p^m)}{p^2}
$$
and 
$$
\sum\limits_{x/2<p\le x} S_{2}\left(p^m\right) = \sum\limits_{b=1}^{x-1} \sum\limits_{\max\{x/2,b\}<p\le x} \frac{\tilde{a}_{E(0,b)}(p^m)}{p^2}.
$$
\end{proof}

From Lemmas \ref{putting} and \ref{S12lemma}, we deduce the following average results.

\begin{lemma} \label{RHS} Let $c,d>0$ and $d_2>d_1>0$ be arbitrary but fixed. Then the following hold. \medskip\\
{\rm (i)} Let $m\in \mathbb{N}$. Then, unconditionally, we have
\begin{equation*}
\sum\limits_{x/2< p\le x} S(p^m) = O_c\left(mx^{1/2}(\log x)^{-c}\right).
\end{equation*}
{\rm (ii)} Let $m\in \mathbb{N}$. Assume that $\log m\le d\log x$. Then, under {\rm MRH}, we have
\begin{equation*}
\sum\limits_{x/2< p\le x} S(p^m) = O_d\left(m\log x\right).
\end{equation*}
{\rm (iii)} Assume that $d_1\log x\le \log M \le d_2 \log x$. Then, under {\rm Hypotheses 1} and {\rm 2}, we have
\begin{equation*}
\sum\limits_{1\le m\le M} \frac{1}{m} \cdot \left| \sum\limits_{x/2< p\le x} S(p^m)\right| = O_{c,d_1,d_2}\left(M(\log x)^{-c}\right).
\end{equation*}
{\rm (iv)} Let $m_1,m_2\in \mathbb{N}$. Assume that $\log m_{1,2}\le d\log x$. Then 
\begin{equation*} 
\sum\limits_{x/2< p\le x} S(p^{m_1})S(p^{m_2}) = O_{c,d}\left(m_1m_2(\log x)^{-c}\right).
\end{equation*}
{\rm (v)} Let $r\ge 2$ and $m_1,...,m_r\in \mathbb{N}$. Then 
\begin{equation} \label{Delicon}
\sum\limits_{x/2< p\le x} S(p^{m_1})\cdots S(p^{m_r}) = O_r\left(\frac{m_1\cdots m_r}{x^{r/2-1}\log x}\right).
\end{equation}
\end{lemma}

\begin{proof} The claimed bounds in (i)-(iv) follow from \eqref{Sdiv}, Lemma \ref{putting}, Lemma \ref{S12lemma} 
and Lemmas \ref{avesiguncond}, \ref{avetwosiguncond}, \ref{avesig} as well as
Hypotheses 1,2, respectively, using partial summation. To prove part (v), we use 
\eqref{trdef}, Theorem \ref{dim}, \eqref{Sdiv}, Lemmas \ref{putting} and \ref{S12lemma}(i) and the Deligne bound (see \cite{Del}) 
$$
|a_f(p)| \le 2p^{(m+1)/2} \quad \mbox{if } f\in \mathcal{F}_{1,m+2}
$$
to obtain
\begin{equation} \label{Spm}
S(p^{m})\ll \frac{\sharp \mathcal{F}_{1,m+2}}{\sqrt{p}} +\frac{m}{p}
\ll \frac{m}{p^{1/2}} 
\end{equation}
for all primes $p\ge 5$ and $m\in \mathbb{N}$, from which the claimed bound \eqref{Delicon} follows using 
$$
\tilde{\pi}(x) \sim \frac{x}{2\log x}
$$
by the prime number theorem.
\end{proof}

We remark that the power savings obtained in Theorem \ref{momenttheorem} depend on the fact that we have a nontrivial estimate
for $S_0\left(p^m\right)$ above, with a saving by a factor of $\sqrt{p}$ over the trivial bound $S_0\left(p^m\right)=O(m)$. 

In addition, we record the following bound for the average of $S_0(p^m)$ which can be proved in the same way as Lemma \ref{RHS}(iii)
above, where Hypothesis 2 is not required.

\begin{lemma} \label{record} Let $c>0$ and $d_2>d_1>0$ be arbitrary but fixed. 
Assume that $d_1\log x\le \log M \le d_2 \log x$. Then, under {\rm Hypotheses 1}, we have
\begin{equation*}
\sum\limits_{1\le m\le M} \frac{1}{m} \cdot \left| \sum\limits_{x/2< p\le x} S_0(p^m)\right| = O_{c,d_1,d_2}\left(M(\log x)^{-c}\right).
\end{equation*}
\end{lemma}

\section{Proof of Theorem \ref{T2}} \label{finalproof}
Now we are ready to prove Theorem \ref{T2}, a key result in this paper. 

\subsection{Opening up the $t$-th power}
Throughout the sequel, we write
$$
X_t:=\frac{1}{4AB}  \sum\limits_{|a|\le A} \sum\limits_{|b|\le B} \Bigg(\sum\limits_{1\le m\le M} U(m) \sum\limits_{\substack{x/2<p\le x\\
p\nmid ab\Delta(a,b)}} 
\tilde{a}_{E(a,b)}(p^m) \Bigg)^t.
$$
Opening the $t$-power, we get
\begin{equation} \label{one}
\begin{split}
X_t
= & \frac{1}{4AB} \sum\limits_{|a|\le A} \sum\limits_{|b|\le B} \sum\limits_{1\le m_1,...,m_t\le M} U(m_1)\cdots U(m_t)\times\\ & 
\sum\limits_{\substack{x/2<p_1,...,p_t\le x\\ (ab\Delta(a,b),p_1\cdots p_t)=1}}  \tilde{a}_{E(a,b)}(p_1^{m_1}) \cdots \tilde{a}_{E(a,b)}(p_t^{m_t}).
\end{split}
\end{equation}
Further, we write
\begin{equation} \label{two}
\begin{split}
& \sum\limits_{\substack{x/2<p_1,...,p_t\le x\\ (ab\Delta(a,b),p_1\cdots p_t)=1}} \tilde{a}_{E(a,b)}(p_1^{m_1}) \cdots \tilde{a}_{E(a,b)}(p_t^{m_t})\\
= & \sum\limits_{u=1}^t \sum\limits_{\{1,...,t\}=\mathcal{S}_1\dot\cup\cdots \dot\cup \mathcal{S}_u} 
\sum\limits_{\substack{x/2< p_1,...,p_u\le x\\(ab\Delta(a,b),p_1\cdots p_u)=1\\ p_r\not = p_s \mbox{\scriptsize\ if } 1\le r<s\le u}} 
\prod\limits_{j=1}^u \prod\limits_{i\in \mathcal{S}_j} \tilde{a}_{E(a,b)}\left(p_j^{m_i}\right),
\end{split}
\end{equation}
where the second sum on the right-hand side runs over all partitions of the set $\{1,...,t\}$ into $u$ disjoint non-empty sets 
$\mathcal{S}_1,...,\mathcal{S}_u$. 

\subsection{Applying Lemma \ref{prodpolys}}
Using Lemma \ref{prodpolys}, we have
\begin{equation*}
\prod\limits_{i\in \mathcal{S}_j} \tilde{a}_{E(a,b)}\left(p_j^{m_i}\right) = \sum\limits_{m=0}^{\infty} 
D\left((m_i)_{i\in \mathcal{S}_j}; m\right) \tilde{a}_{E(a,b)}\left(p_j^m\right) 
\end{equation*} 
for all $j\in \{1,...,u\}$ if $(\Delta(a,b),p_1\cdots p_u)=1$.
From this and Theorem \ref{facts}(iii), we further deduce that 
\begin{equation} \label{third}
\prod\limits_{j=1}^u \prod\limits_{i\in \mathcal{S}_j} a_{E(a,b)}\left(p_j^{m_i}\right)= 
\sum\limits_{\alpha_1=0}^{\infty} \cdots \sum\limits_{\alpha_u=0}^{\infty} 
\left(\prod\limits_{j=1}^u  D\left((m_i)_{i\in \mathcal{S}_j}; \alpha_j\right)\right) 
\tilde{a}_{E(a,b)}\left(p_1^{\alpha_1}\cdots p_u^{\alpha_u}\right)
\end{equation}
under this condition. Combining \eqref{one}, \eqref{two} and \eqref{third}, and rearranging summations, we obtain
\begin{equation} \label{fourth}
\begin{split}
 X_t
= &  \sum\limits_{u=1}^t 
\sum\limits_{\alpha_1=0}^{\infty} \cdots \sum\limits_{\alpha_u=0}^{\infty} 
C(\alpha_1,...,\alpha_u) \sum\limits_{\substack{x/2< p_1,...,p_u\le x\\  p_r\not = p_s \mbox{\scriptsize\ if } 1\le r<s\le u}}
\frac{1}{4AB}\times\\ & \mathop{\sum\limits_{|a|\le A} \sum\limits_{|b|\le B}}_{(ab\Delta(a,b),p_1\cdots p_u)=1} \tilde{a}_{E(a,b)}\left(p_1^{\alpha_1}\cdots p_u^{\alpha_u}\right),
\end{split}
\end{equation}
where 
\begin{equation}
 C(\alpha_1,...,\alpha_u):=\sum\limits_{\{1,...,t\}=\mathcal{S}_1\dot\cup\cdots \dot\cup \mathcal{S}_u} 
 \sum\limits_{1\le m_1,...,m_t\le M} U(m_1)\cdots U(m_t) \prod\limits_{j=1}^u 
 D\left(\left(m_i\right)_{i\in \mathcal{S}_j};\alpha_j\right).
\end{equation}

\subsection{Estimation of $C(\alpha_1,...,\alpha_t)$}
Let 
\begin{equation} \label{zn}
z:=\sharp\{i\in \{1,...,u\} :\alpha_i=0\} \quad \mbox{and} \quad n:=\sharp\{i\in \{1,...,u\} : \alpha_i\not=0\}.
\end{equation}
Then from Lemma \ref{prodpolys} and $U(m_i)\ll 1/m_i$, we deduce that
\begin{eqnarray}
\label{Calphaest}
C(\alpha_1,...,\alpha_u) & = & O_t\left(\frac{M^{t-2z-n}(\log M)^{t-u}}{(\alpha_1+1)\cdots (\alpha_u+1)}\right) \quad \mbox{ if } 2z+n\le t\\
\label{toomuch1}
C(\alpha_1,...,\alpha_u) & = & 0 \quad \mbox{ if } 2z+n>t \\  
\label{000}
C(0,...,0) & = & O_t\left(M^{t-2z-1}(\log M)^{t-z}\right)\quad \mbox{ if } 2z<t \\   
\label{000ext}
C(0,...,0)& = & \frac{(2z)!}{2^zz!}\cdot Z^z\quad \mbox{ if } 2z=t \\  
\label{toomuch}
C(\alpha_1,...,\alpha_u) & = & 0 \quad \mbox{ if } \alpha_i>tM \mbox{ for an } i\in \{1,...,u\},  
\end{eqnarray}
where $Z$ is defined as in \eqref{Zdef}. Here we use the first two equations in \eqref{Dbounds} if $\alpha_i\not=0$ and the last three equations in
\eqref{Dbounds} if $\alpha_i=0$. 

\subsection{Averaging over $a$ and $b$}
Using Lemma \ref{mult} and Theorem \ref{improvedav}, we have 
\begin{equation*}
\begin{split}
& \frac{1}{4AB}  
\mathop{\sum\limits_{|a|\le A} \sum\limits_{|b|\le B}}_{(ab\Delta(a,b),p_1\cdots p_u)=1} \tilde{a}_{E(a,b)}\left(p_1^{\alpha_1}\cdots p_u^{\alpha_u}\right)\\
= & S\left(p_1^{\alpha_1}\right)\cdots S\left(p_u^{\alpha_u}\right) + O_u\left(
\prod\limits_{i=1}^u \left(\alpha_i+1\right)\cdot x^{u/2+\varepsilon}\left(\frac{1}{A}+\frac{1}{B}\right)\right).
\end{split}
\end{equation*}
Combining this with \eqref{fourth}, and using \eqref{zn}, \eqref{Calphaest}, \eqref{toomuch1} and \eqref{toomuch}, we obtain
\begin{equation} \label{fifth}
\begin{split}
X_t = &  \sum\limits_{u=1}^t 
\sum\limits_{\alpha_1=0}^{tM} \cdots \sum\limits_{\alpha_u=0}^{tM} 
C(\alpha_1,...,\alpha_u) \sum\limits_{\substack{x/2< p_1,...,p_u\le x\\ p_r\not = p_s \mbox{\scriptsize\ if } 1\le r<s\le u}}
S\left(p_1^{\alpha_1}\right)\cdots S\left(p_u^{\alpha_u}\right) +\\ & 
O_t\left(M^{t+\varepsilon}x^{3t/2+\varepsilon}\left(\frac{1}{A}+\frac{1}{B}\right)\right).
\end{split}
\end{equation}

\subsection{Separating the primes}
Next, we remove the summation conditon $p_r\not=p_s$ which was introduced to make use of the multiplicativity of the functions
$\tilde{a}_{E(a,b)}(n)$
and $S(n)$. In this
way, we make the prime variables $p_j$ independent. 

We first observe that
\begin{equation} \label{someobs}
\begin{split}
\sum\limits_{\substack{x/2< p_1,...,p_u\le x\\ p_r\not = p_s \mbox{\scriptsize\ if } 1\le r<s\le u}} 
S\left(p_1^{\alpha_1}\right)
\cdots S\left(p_u^{\alpha_u}\right) = & (\tilde{\pi}(x)-n)\cdots (\tilde{\pi}(x)-u+1) \times\\ & 
\sum\limits_{\substack{x/2<p_1,...,p_n\le x\\ p_r\not=p_s \mbox{\scriptsize\ if } 1\le r<s\le n}}
S\left(p_1^{\beta_1}\right)
\cdots S\left(p_n^{\beta_n}\right),
\end{split}
\end{equation}
where $z$ and $n$ are defined as in \eqref{zn}, and $(\beta_1,...,\beta_n)$ is the $n$-tuple obtained by removing all zero elements from the
$u$-tuple $(\alpha_1,...,\alpha_u)$. To see \eqref{someobs}, we note that if $n$ distinct primes $p_1,...,p_n$ in $(x/2,x]$ are fixed, then there 
are $\tilde{\pi}(x)-n$ possibilities to choose a prime $p_{n+1}$ in $(x/2,x]$ distinct from $p_1,...,p_n$, $\tilde{\pi}(x)-n-1$ possibilities 
to choose another prime $p_{n+2}$ in $(x/2,x]$ distinct from $p_1,...,p_{n+1}$, and so on. We further note that
\begin{equation} \label{somenote}
(\tilde{\pi}(x)-n)\cdots (\tilde{\pi}(x)-u+1)=\tilde{\pi}(x)^z+O\left(\tilde{\pi}(x)^{z-1}\right).
\end{equation}

Now we claim that the right-hand side of \eqref{someobs} can be written as a sum over partitions $\mathcal{P}$ of the set $\{1,...,n\}$ in 
the form
\begin{equation} \label{three} 
\sum\limits_{\substack{x/2< p_1,...,p_n\le x\\ p_r\not = p_s \mbox{\scriptsize\ if } 1\le r<s\le n}} 
S\left(p_1^{\beta_1}\right)
\cdots S\left(p_n^{\beta_n}\right)
= \sum\limits_{\substack{\mathcal{P}=\{\mathcal{M}_1,...,\mathcal{M}_k\}\\ \mathcal{M}_1\dot{\cup}\cdots \dot{\cup} \mathcal{M}_k=\{1,...,n\}}} 
 A(\mathcal{P})\cdot \prod\limits_{l=1}^k \sum\limits_{x/2< p\le x} \prod\limits_{j\in \mathcal{M}_l} S\left(p^{\beta_j}\right),
\end{equation}
where
\begin{equation} \label{thecoeff}
A(\mathcal{P})=A(\{\mathcal{M}_1,...,\mathcal{M}_k\})=(-1)^{n-k}\cdot \prod\limits_{l=1}^k \left(\sharp \mathcal{M}_l-1\right)!.
\end{equation}
To establish \eqref{three}, we use the identity 
\begin{equation}
\begin{split}
\sum\limits_{\substack{x/2<p_1,...,p_n\le x\\ p_r\not = p_s \mbox{\scriptsize\ if } 1\le r<s\le n}} S\left(p_1^{\beta_1}\right)
\cdots S\left(p_n^{\beta_n}\right) = &
\sum\limits_{\substack{x/2< p_1,...,p_n\le x\\ p_r\not = p_s \mbox{\scriptsize\ if } 1\le r<s\le n-1}} S\left(p_1^{\beta_1}\right)
\cdots S\left(p_n^{\beta_n}\right)-\\
& \sum\limits_{w=1}^{n-1} \sum\limits_{\substack{x/2< p_1,...,p_n\le x\\ p_r\not = p_s \mbox{\scriptsize\ if } 1\le r<s\le n-1\\
p_{n}=p_w}} S\left(p_1^{\beta_1}\right)
\cdots S\left(p_n^{\beta_n}\right)
\end{split}
\end{equation}
for $n\ge 2$ and proceed by induction over $n$. In this way, we see that 
\begin{equation*} 
\begin{split}
& \sum\limits_{\substack{x/2< p_1,...,p_n\le x\\ p_r\not = p_s \mbox{\scriptsize\ if } 1\le r<s\le n}} 
S\left(p_1^{\beta_1}\right)
\cdots S\left(p_n^{\beta_n}\right)\\
= & \sum\limits_{\substack{\mathcal{P}'=\{\mathcal{M}_1,...,\mathcal{M}_k\}\\ \mathcal{M}_1\dot{\cup}\cdots \dot{\cup} \mathcal{M}_k=\{1,...,n-1\}}} 
A(\mathcal{P}')\cdot \left(\prod\limits_{l=1}^k \sum\limits_{x/2< p\le x} \prod\limits_{j\in \mathcal{M}_l} S\left(p^{\beta_j}\right)\right) \cdot 
\left(\sum\limits_{x/2<p\le x} S\left(p^{\beta_n}\right)\right)-\\
& \sum\limits_{\substack{\mathcal{P}'=\{\mathcal{M}_1,...,\mathcal{M}_k\}\\ \mathcal{M}_1\dot{\cup}\cdots \dot{\cup} \mathcal{M}_k=\{1,...,n-1\}}} 
A(\mathcal{P}')\cdot \sum\limits_{m=1}^k \left(\sharp \mathcal{M}_m\right)\cdot \left(
\prod\limits_{\substack{l=1\\ l\not=m}}^k \sum\limits_{x/2< p\le x} \prod\limits_{j\in \mathcal{M}_l} S\left(p^{\beta_j}\right)\right) \times\\ &
\left(\sum\limits_{x/2<p\le x} S\left(p^{\beta_n}\right) \cdot \prod\limits_{j\in \mathcal{M}_m} S\left(p^{\beta_j}\right)\right)\\
= & \sum\limits_{\substack{\mathcal{P}=\{\mathcal{M}_1,...,\mathcal{M}_k\}\\ \mathcal{M}_1\dot{\cup}\cdots \dot{\cup} \mathcal{M}_k=\{1,...,n\}}} 
 A(\mathcal{P})\cdot \prod\limits_{l=1}^k \sum\limits_{x/2< p\le x} \prod\limits_{j\in \mathcal{M}_l} S\left(p^{\beta_j}\right),
\end{split}
\end{equation*}
where the coefficients satisfy the recursive relations
\begin{equation*}
\begin{split}
A(\{\{1\}\})= & 1, \\
A\left(\mathcal{P}'\dot{\cup} \{\{n\}\}\right)= & A\left(\mathcal{P}'\right),\\
A\left(\left\{\mathcal{M}_1,...,\mathcal{M}_{m-1},\mathcal{M}_m\cup \{n\},\mathcal{M}_{m+1},...,\mathcal{M}_k\right\}\right)= & 
-\left(\sharp \mathcal{M}_m\right)\cdot A\left(\mathcal{P}'\right)
\end{split}
\end{equation*}
if $\mathcal{P}'=\{\mathcal{M}_1,...,\mathcal{M}_k\}$ and $\mathcal{M}_1\dot{\cup}\cdots \dot{\cup} \mathcal{M}_k=\{1,...,n-1\}$. These recursive relations imply the explicit formula 
\eqref{thecoeff} via another induction argument. 

\subsection{Estimation of $X_t$ under MRH}
Throughout the following subsections, let $F,c,\varepsilon>0$ be arbitrary but fixed
constants and $0\le \beta_i\le tM$ for $i\in \{1,...,n\}$.

We first estimate $X_t$ under MRH. Assume that
\begin{equation} \label{MRHcond}
\tilde{\pi}(x)^{1/2} \le M\le x^F\quad \mbox{ and } \quad
A,B\ge x^{3t/2+(F+2)\varepsilon}.
\end{equation}
Then using parts (ii) and (v) of Lemma \ref{RHS},
we deduce from \eqref{three} that, under MRH,
\begin{equation} \label{obtain} 
\sum\limits_{\substack{x/2< p_1,...,p_n\le x\\ p_r\not = p_s \mbox{\scriptsize\ if } 1\le r<s\le n}} S\left(p_1^{\beta_1}\right)\cdots
S\left(p_n^{\beta_n}\right)
\ll_{n,F} \beta_1\cdots \beta_n (\log x)^n.
\end{equation}
Combining \eqref{Calphaest}, \eqref{toomuch1}, \eqref{fifth}, \eqref{someobs}, \eqref{somenote}, \eqref{MRHcond} and \eqref{obtain},
we obtain
\begin{equation*} \label{ZtMRH}
X_t=O_{t,F,\varepsilon}\left(M^t(\log x)^{t}\right) \quad \mbox{under MRH}.
\end{equation*} 
 
\subsection{Unconditonal estimation of $X_t$}
Next, we estimate $X_t$ unconditionally in a similar way. Assume that
\begin{equation} \label{uncondcond}
x^{\varepsilon}\le M\le x^F\quad \mbox{ and } \quad
A,B\ge x^{t+(F+2)\varepsilon}.
\end{equation}
Then using parts (i) and (v) of Lemma \ref{RHS}, we deduce from \eqref{three} that
\begin{equation} \label{obtain1} 
\sum\limits_{\substack{x/2< p_1,...,p_n\le x\\ p_r\not = p_s \mbox{\scriptsize\ if } 1\le r<s\le n}} S\left(p_1^{\beta_1}\right)\cdots
S\left(p_n^{\beta_n}\right)
\ll_{n,F,c} \beta_1\cdots \beta_n x^{n/2}(\log x)^{-(t+c)}.  
\end{equation}
Combining \eqref{Calphaest}, \eqref{toomuch1}, \eqref{fifth}, \eqref{someobs}, \eqref{somenote}, \eqref{uncondcond} and \eqref{obtain1}, 
we obtain  
\begin{equation*} \label{Ztuncond}
X_t=O_{t,F,c,\varepsilon}\left(M^tx^{t/2}(\log x)^{-c}\right) \quad \mbox{unconditionally}.
\end{equation*} 

\subsection{Estimation of $X_t$ under Hypotheses 1,2}
Finally, we estimate $X_t$ under Hypotheses 1 and 2. Assume that
\begin{equation} \label{Hypocond}
\tilde{\pi}(x)^{1/2} \le M\le x^F\quad \mbox{ and } \quad
A,B\ge x^{3t/2+(F+2)\varepsilon}.
\end{equation} 
Using \eqref{three} and the triangle inequality, we have
\begin{equation} \label{threeav} 
\begin{split}
& \sum\limits_{1\le \beta_1,...,\beta_n\le tM} \frac{1}{\beta_1\cdots \beta_n}\cdot \left|
\sum\limits_{\substack{x/2< p_1,...,p_n\le x\\ p_r\not = p_s \mbox{\scriptsize\ if } 1\le r<s\le n}} 
S\left(p_1^{\beta_1}\right) \cdots S\left(p_n^{\beta_n}\right) \right|\\
\le & \sum\limits_{\substack{\mathcal{P}=\{\mathcal{M}_1,...,\mathcal{M}_k\}\\ \mathcal{M}_1\dot{\cup}\cdots \dot{\cup} \mathcal{M}_k=\{1,...,n\}}} 
 \left|A(\mathcal{P})\right|\cdot \prod\limits_{l=1}^k \left| \sum\limits_{x/2< p\le x} \prod\limits_{j\in \mathcal{M}_l} 
 \sum\limits_{1\le \beta_j\le tM} \frac{S\left(p^{\beta_j}\right)}{\beta_j}\right|.
\end{split}
\end{equation}
From Lemma \ref{RHS} and \eqref{threeav}, we deduce that 
\begin{equation} \label{obtain3} 
\sum\limits_{1\le \beta_1,...,\beta_n\le tM} \frac{1}{\beta_1\cdots \beta_n}\cdot \left|
\sum\limits_{\substack{x/2< p_1,...,p_n\le x\\ p_r\not = p_s \mbox{\scriptsize\ if } 1\le r<s\le n}} 
S\left(p_1^{\beta_1}\right) \cdots S\left(p_n^{\beta_n}\right) \right|
\ll_{n,F,c,t} M^n  (\log x)^{-(t+c)},
\end{equation}
where we use parts (iii) and (v) of Lemma \ref{RHS} if $n>2$, parts (iii) and (iv) if $n=2$, and part (iii) if $n=1$.   
Combining \eqref{Calphaest}, \eqref{toomuch1}, \eqref{000}, \eqref{000ext}, \eqref{fifth}, \eqref{someobs}, \eqref{somenote}, \eqref{Hypocond}, 
and \eqref{obtain3}, we obtain 
\begin{equation*} \label{ZtHypo}
X_t=\delta(t)\cdot \frac{t!}{2^{t/2}(t/2)!}\cdot (Z\tilde{\pi}(x))^{t/2}+O_{t,F,c,\varepsilon}\left(M^t(\log x)^{-c}\right) \quad 
\mbox{under Hypotheses 1 and 2}.
\end{equation*}

Combining the results of the last three subsections, we obtain
claimed estimate \eqref{das} upon changing the term $(F+2)\varepsilon$ in the conditions on $A$ and $B$ into $\varepsilon$.
This completes the proof.
\begin{flushright} $\Box$ \end{flushright}

\section{Proof of Theorem \ref{T3}} \label{finalproof1}
We recall that the summation condition $p\nmid ab$, which was present on the left-hand side of \eqref{das}, is omitted in \eqref{dasi}. 
To prove Theorem \ref{T3}, we proceed in the same way as in the proof of Theorem \ref{T2}, where we replace the estimate \eqref{thebound} by 
\begin{equation} \label{due}
\mathop{\sum\limits_{|a|\le A} \sum\limits_{|b|\le B}}_{(\Delta(a,b),n)=1} \tilde{a}_{E(a,b)}(n) = 4AB S_0(n) + 
O\left(d(n)s(n)^2\right)+O\left(d(n)s(n)(A+B)\right),  
\end{equation}
which can be established in a similar way as \eqref{weakav}. We further use the fact that $S_0(n)$ is multiplicative just like 
$S(n)$ is multiplicative by Lemma \ref{mult} and employ Lemma \ref{record}, which holds under Hypothesis 1 only, 
instead of Lemma \ref{RHS}(iii).
Avoiding Hypothesis 2 comes at the cost of replacing the condition $A,B\ge x^{3t/2+\varepsilon}$ in the third estimate on 
the right-hand side of \eqref{das} by the stronger condition 
$A,B\ge x^{2t+\varepsilon}$, which is due to the weaker $O$-term in \eqref{due} in place of the $O$-term in \eqref{thebound}. 
\begin{flushright} $\Box$ \end{flushright}

\section{Proof of Theorem \ref{momenttheorem}}
Let $M\in \mathbb{N}$, to be fixed later.

\subsection{Removing the primes $p$ dividing $ab$} \label{rem}
We start our proof with getting rid of the contribution of primes $p$ dividing $ab$. First, we separate the contribution of $ab=0$, 
observing that 
\begin{equation} \label{trivialissimo}
\begin{split}
& \frac{1}{4AB} \sum\limits_{|a|\le A} \sum\limits_{|b|\le B} \left(N_I(E(a,b),x) - \tilde{\pi}(x)\mu(I)\right)^t\\
= & \frac{1}{4AB} 
\sum\limits_{0<|a|\le A} \sum\limits_{0<|b|\le B} \left(N_I(E(a,b),x) - \tilde{\pi}(x)\mu(I)\right)^t
+O\left(x^t\left(\frac{1}{A}+\frac{1}{B}\right)\right)
\end{split}
\end{equation}
by a trivial estimation. Now it suffices to treat the case $ab\not=0$. We recall the definitions of $U^{\pm}_I(m)$ and $P_I^{\pm}(E,x)$ 
in Theorem \ref{errorapprox} and deduce the bound 
\begin{equation} \label{telescop}
\sum\limits_{1\le m\le M}  U^{\pm}_I(m) \tilde{a}_E(p^m) = O(\log 2M) 
\end{equation} 
for every prime $p$ of good reduction at $E$ from the bound 
$$
\tilde{a}_E(p^{m+2})-\tilde{a}_E(p^m)=O(1),
$$
which follows from the fact that $\tilde{a}_E(p^{m+2})-\tilde{a}_E(p^m)= 2\cos(m\theta_E(p))$ where $\tilde{a}_E(p)= 2\cos(\theta_E(p))$. For details, see Corollary 2.3 of \cite{PrS}.
From \eqref{ap}, \eqref{telescop} and the well-known bound $\omega(ab)\ll \log(2|ab|)$ for the number of prime
divisors of $ab$, we deduce that
\begin{equation} \label{apnew}
\begin{split}
& \tilde{P}_I^{-}(E(a,b),x)+O\left(\frac{\tilde{\pi}(x)}{M}+\log(2|ab|)\log(2M)\right)\\ \le & 
N_I(E(a,b),x) - \tilde{\pi}(x)\mu(I) \le \tilde{P}_I^+(E(a,b),x)+ O\left(\frac{\tilde{\pi}(x)}{M}+\log(2|ab|)\log(2M)\right)
\end{split}
\end{equation} 
if $ab\not=0$, where 
\begin{equation*}
\begin{split}
& \tilde{P}_I^{\pm}(E(a,b),x):= \sum\limits_{1\le m\le M}
U_I^{\pm}(m)\sum\limits_{\substack{x/2<p\le x\\ p\nmid ab\Delta(a,b)}} \tilde{a}_{E(a,b)}(p^m)\\ = &
P_I^{\pm}(E(a,b),x) - \sum\limits_{1\le m\le M}
U_I^{\pm}(m)\sum\limits_{\substack{x/2<p\le x\\ p\nmid \Delta(a,b)\\ p|ab}} \tilde{a}_{E(a,b)}(p^m). 
\end{split}
\end{equation*}

\subsection{Application of the binomial formula}
Next, we use the binomial formula to write 
\begin{equation*}
\begin{split}
\left(N_I(E,x) - \tilde{\pi}(x)\mu(I)\right)^t= \sum\limits_{s=0}^t \binom{t}{s} 
\left(N_I(E,x) - \tilde{\pi}(x)\mu(I)- \tilde{P}_I^-(E,x)\right)^{t-s}\tilde{P}_I^-(E,x)^s.
\end{split}
\end{equation*}
Using the Cauchy-Schwarz inequality and taking into acoount that $p|ab\Delta(a,b)$ for every prime $p$ if $a=0$ or $b=0$, we deduce that
\begin{equation} \label{laststep} 
\begin{split}
& \frac{1}{4AB} \sum\limits_{0<|a|\le A} \sum\limits_{0<|b|\le B} \left(N_I(E(a,b),x) - \tilde{\pi}(x)\mu(I)\right)^t\\
= & \frac{1}{4AB} \sum\limits_{|a|\le A} \sum\limits_{|b|\le B} \tilde{P}_I^-(E(a,b),x)^t+\\ &
O_t\left(\frac{1}{4AB} \sum\limits_{|a|\le A} \sum\limits_{|b|\le B} \left(N_I(E(a,b),x) - \tilde{\pi}(x)\mu(I)- \tilde{P}_I^-(E(a,b),x)\right)^{t}
+ \right.  \\ & \left.  \sum\limits_{s=1}^{t-1}\left(
\frac{1}{4AB} \sum\limits_{|a|\le A} \sum\limits_{|b|\le B} \tilde{P}_I^-(E(a,b),x)^{2s}\right)^{1/2}\times  \right. \\ &  
\left. \left(
\frac{1}{4AB} \sum\limits_{|a|\le A} \sum\limits_{|b|\le B} \left(N_I(E(a,b),x) - \tilde{\pi}(x)\mu(I)- \tilde{P}_I^-(E(a,b),x)\right)^{2(t-s)}
\right)^{1/2}
\right).
\end{split}
\end{equation} 

\subsection{Estimation of moments}
We need to evaluate the $v$-th moment of $P_I^-(E(a,b),x)$ for $v=2s$ and $v=t$. 
Applying Theorems \ref{amazing} and \ref{T2}, we get 
\begin{equation} \label{firstvbound}
\begin{split}
& \frac{1}{4AB} \sum\limits_{|a|\le A} \sum\limits_{|b|\le B} \tilde{P}_I^-(E(a,b),x)^v\\ = &
\delta(v)\cdot \frac{v!}{2^{v/2}(v/2)!}\cdot \left(\mu(I)-\mu(I)^2+\frac{\log(2M)}{M}\right)^{v/2}
\left(\tilde{\pi}(x)^{v/2}+O\left(\tilde{\pi}(x)^{v/2-1}\right)\right) + \\ &  
\begin{cases}
O_{v,F,c,\varepsilon}\left(M^vx^{v/2}(\log x)^{-c}\right) & \mbox{unconditionally if } 
x^{\varepsilon} \le M\le x^F \mbox{ and } A,B\ge x^{v+\varepsilon}\\
O_{v,F,\varepsilon}\left(M^v(\log x)^v\right) & \mbox{under {\rm MRH} if } \tilde{\pi}(x)^{1/2}\le M\le x^F \mbox{ and } A,B\ge 
x^{3v/2+\varepsilon}\\
O_{v,F,c,\varepsilon}\left(M^v(\log x)^{-c}\right) & \mbox{under {\rm Hyp.1,2}  if } 
\tilde{\pi}(x)^{1/2}\le M\le x^F \mbox{ and } A,B\ge x^{3v/2+\varepsilon}.
\end{cases}
\end{split}
\end{equation}

We further need to evaluate the $v$-th moments of $N_I(E(a,b),x) - \tilde{\pi}(x)\mu(I)- \tilde{P}_I^-(E(a,b),x)$ for $v=2(t-s)$, which is even, 
and $v=t$, which
is possibly odd. Using \eqref{apnew}, we observe that
\begin{equation} \label{inequal}
\begin{split}
-\frac{K\tilde{\pi}(x)}{M}\le & N_I(E(a,b),x) - \tilde{\pi}(x)\mu(I)- \tilde{P}_I^-(E(a,b),x)\\ \le &
\frac{L\tilde{\pi}(x)}{M} + \tilde{P}_I^+(E(a,b),x)-P_I^-(E(a,b),x)
\end{split}
\end{equation}
for some absolute constants $K,L>0$, provided that
\begin{equation} \label{somecond}
\log(2|ab|)\log(2M) \le \log(2AB)\log(2M)\le \frac{\tilde{\pi}(x)}{M},
\end{equation}
which we want to assume from now on. We further note that 
$$
\tilde{P}_I^+(E(a,b),x)-\tilde{P}_I^-(E(a,b),x)=\sum\limits_{1\le m\le M}
\left(U_I^{+}(m)-U_I^{-}(m)\right)\sum\limits_{\substack{x/2<p\le x\\ p\nmid ab\Delta(a,b)}} \tilde{a}_{E(a,b)}(p^m)
$$
and 
$$
U_I^{+}(m)-U_I^{-}(m)\ll \frac{1}{M} \quad \mbox{for } 1\le m\le M.
$$
If $v$ is even, then, using the above considerations and the inequality
$$
\left(N_I(E(a,b),x) - \tilde{\pi}(x)\mu(I)- \tilde{P}_I^-(E(a,b),x)\right)^v \ll_v \left(\frac{\tilde{\pi}(x)}{M}\right)^v + 
\left(\tilde{P}_I^+(E(a,b),x)-\tilde{P}_I^-(E(a,b),x)\right)^v
$$
following from \eqref{inequal}, we deduce from Theorem \ref{T2} with $F=1$ that
\begin{equation} \label{secondvbound}
\begin{split}
& \frac{1}{4AB} \sum\limits_{|a|\le A} \sum\limits_{|b|\le B} \left(N_I(E(a,b),x) - \tilde{\pi}(x)\mu(I)- \tilde{P}_I^-(E(a,b),x)\right)^v\\
= & O_v\left(\frac{\tilde{\pi}(x)^{v/2}\log(2M)}{M}+\left(\frac{\tilde{\pi}(x)}{M}\right)^v\right)+\\ &  
\begin{cases}
O_{v,c,\varepsilon}\left(M^vx^{v/2}(\log x)^{-c}\right) & \mbox{unconditionally if } 
x^{\varepsilon} \le M\le \tilde{\pi}(x) \mbox{ and } A,B\ge x^{v+\varepsilon}\\
O_{v,\varepsilon}\left(M^v(\log x)^v\right) & \mbox{under {\rm MRH} if } \tilde{\pi}(x)^{1/2}\le M\le \tilde{\pi}(x) \mbox{ and } A,B\ge 
x^{3v/2+\varepsilon}\\
O_{v,c,\varepsilon}\left(M^v(\log x)^{-c}\right) & \mbox{under {\rm Hyp.1,2}  if } 
\tilde{\pi}(x)^{1/2}\le M\le \tilde{\pi}(x) \mbox{ and } A,B\ge x^{3v/2+\varepsilon}.
\end{cases}
\end{split}
\end{equation}
If $v$ is odd, then we need to argue more carefully. Here we use the fact that 
\begin{equation} \label{inequal1}
\begin{split}
-\left(\frac{K\tilde{\pi}(x)}{M}\right)^v\le & \left(N_I(E(a,b),x) - \tilde{\pi}(x)\mu(I)- \tilde{P}_I^-(E(a,b),x)\right)^v\\ \le & 
\left(\frac{L\tilde{\pi}(x)}{M} + \tilde{P}_I^+(E(a,b),x)-\tilde{P}_I^-(E(a,b),x)\right)^v\\ = & \sum\limits_{s=0}^v \binom{v}{s} 
\left(\frac{L\tilde{\pi}(x)}{M}\right)^{v-s} \left(\tilde{P}_I^+(E(a,b),x)-\tilde{P}_I^-(E(a,b),x)\right)^s, 
\end{split}
\end{equation}
which follows from \eqref{inequal} as well.
Applying Theorem \ref{T2} again, we see after a short calculation that the same bound as in \eqref{secondvbound} holds in this case too.

\subsection{Completion of the proof}
Now we choose
$$
M:=\begin{cases} \left\lceil x^{1/4}(\log x)^{c/(2t)}\right\rceil
& \mbox{unconditionally} \\ \left\lceil \tilde{\pi}(x)^{1/2}\right\rceil & \mbox{under MRH}\\ \left\lceil x^{1/2}(\log x)^{c/(2t)} \right\rceil & 
\mbox{under Hypotheses 1,2} \end{cases}
$$
and impose the condition that
$$
AB\le \exp\left(x^{1/2-\varepsilon}\right)
$$
so that \eqref{somecond} is satisfied in each case if $x$ is large enough.  
Then, if $1\le v\le t$, \eqref{firstvbound} simplifies into
\begin{equation} \label{firstvboundsimp}
\begin{split}
& \frac{1}{4AB} \sum\limits_{|a|\le A} \sum\limits_{|b|\le B} \tilde{P}_I^-(E(a,b),x)^v =  
\delta(v)\cdot \frac{v!}{2^{v/2}(v/2)!}\cdot \tilde{\pi}(x)^{v/2}\left(\mu(I)-\mu(I)^2\right)^{v/2} + \\ &  
\begin{cases}
O_{v,c,\varepsilon}\left(x^{3v/4}(\log x)^{-c/2}\right) & \mbox{unconditionally if } A,B\ge x^{v+\varepsilon}\\
O_{v,\varepsilon}\left(\tilde{\pi}(x)^{v/2}(\log x)^v\right) & \mbox{under {\rm MRH} if } A,B\ge 
x^{3v/2+\varepsilon}\\
O_{v,c,\varepsilon}\left(x^{v/2}(\log x)^{-c/2}\right) & \mbox{under {\rm Hyp.1,2}  if } A,B\ge x^{3v/2+\varepsilon},
\end{cases}
\end{split}
\end{equation}
and \eqref{secondvbound} simplifies into
\begin{equation} \label{secondvboundsimp}
\begin{split}
& \frac{1}{4AB} \sum\limits_{|a|\le A} \sum\limits_{|b|\le B} \left(N_I(E(a,b),x) - \tilde{\pi}(x)\mu(I)- \tilde{P}_I^-(E(a,b),x)\right)^v\\
= & 
\begin{cases}
O_{v,c,\varepsilon}\left(x^{3v/4}(\log x)^{-c/2}\right) & \mbox{unconditionally if } A,B\ge x^{v+\varepsilon}\\
O_{v,\varepsilon}\left(\tilde{\pi}(x)^{v/2}(\log x)^v\right) & \mbox{under {\rm MRH} if } A,B\ge 
x^{3v/2+\varepsilon}\\
O_{v,c,\varepsilon}\left(x^{v/2}(\log x)^{-c/2}\right) & \mbox{under {\rm Hyp.1,2}  if } A,B\ge x^{3v/2+\varepsilon}.
\end{cases}
\end{split}
\end{equation}

Combining \eqref{laststep}, \eqref{firstvboundsimp} and 
\eqref{secondvboundsimp} gives the first three estimates in \eqref{momentbounds} upon changing $c$ into $2c$ in the unconditional case 
and into $4tc$ under Hypotheses 1, 2. In our computations, we take into account that the main term on the right hand side of (\ref{firstvboundsimp}) may dominate under Hypotheses 1,2. The fourth estimate is established in a similar way, but here we avoid removing the primes $p$ dividing
$ab$, as carried out in subsection \ref{rem}, work directly with the polynomials $P_I^{\pm}(E(a,b),x)$ instead of $\tilde{P}_I^{\pm}(E(a,b),x)$,
and apply Theorem \ref{T3} in place of Theorem \ref{T2}, where the truth of Hypothesis 2 is not assumed. We note that avoiding the 
treatment in subsection \ref{rem} also saves us from assuming that $AB\le \exp\left(x^{1/2-\varepsilon}\right)$.
Finally, we point out that we need to introduce the function $\eta(t)$ in the conditions on $A$ and $B$ in \eqref{momentbounds} 
because of the use of
the Cauchy-Schwarz inequality in \eqref{laststep}. The latter introduces the powers $2s$ and $2(t-s)$ which go up to $2(t-1)$. 
Now the proof of Theorem 
\ref{momenttheorem} is complete.  \begin{flushright} $\Box$ \end{flushright}

\end{document}